\documentclass[11pt,a4paper,twoside,reqno]{amsart}
\usepackage[T1]{fontenc}
\usepackage[english]{babel}
\usepackage{amsmath,amssymb}
\usepackage{bm}
\usepackage{blindtext,varioref}
\usepackage{float}
\usepackage{graphicx}
\usepackage{comment}
\usepackage[left=1.5cm, right=1.5cm, top=2cm, bottom=2cm]{geometry}
\usepackage{float}
\usepackage{graphicx}
\usepackage{color}
\usepackage{epsfig}
\usepackage{caption}
\usepackage{subcaption}
\usepackage{enumerate}
\usepackage{booktabs}
\usepackage{array}
\usepackage{cite}
\usepackage{hyperref}
\usepackage{xfrac}
\usepackage{float}
\usepackage{afterpage}

\usepackage{stmaryrd}
\usepackage{tikz}
\usepackage{siunitx}
\usepackage{float}

\usepackage{algorithm}
\usepackage{algorithmic}
\usepackage{placeins}
\usepackage{amsmath,amssymb,mathrsfs,amsthm}

\theoremstyle{plain}
\newtheorem{theorem}{Theorem}
\newtheorem{lemma}[theorem]{Lemma}

\newtheorem{proposition}[theorem]{Proposition}
\theoremstyle{definition}
\newtheorem{definition}[theorem]{Definition}

\theoremstyle{remark}
\newtheorem{remark}[theorem]{Remark}
\newtheorem{example}[theorem]{Example}

\newcommand{\bmx}{\ensuremath{\bm{x}}}
\newcommand{\bmu}{\ensuremath{\bm{u}}}
\newcommand{\bmU}{\ensuremath{\bm{U}}}
\newcommand{\bmn}{\ensuremath{\bm{n}}}
\newcommand{\kfourier}{\ensuremath{\bm{\kappa}}}
\newcommand{\mudiff}{\ensuremath{\hat{\mu}}}
\newcommand{\normeuclide}[1]{\begin{vmatrix}{#1}\end{vmatrix}_2}
\definecolor{mygreen}{rgb}{0.1333, 0.5451, 0.1333}
\newcommand{\normLdeux}[1]{\left\lVert #1 \right\rVert_{L^2}}
\newcommand{\Ux}{\ensuremath{U_x}}
\newcommand{\kUx}{\ensuremath{\kappa_x U_x}}
\newcommand{\Rd}{\mathbb{R}^d}
\DeclareMathOperator{\Moperator}{M}

\title[A mathematical study  of navigation  in stratified waters]{Mathematical and numerical study of a model for navigation  in stratified waters}

\author[Z. Rammal]{Zeina Rammal$^\dag$}
\address{$^\dag$  Laboratoire de Math\'ematiques et Applications (UMR CNRS 7348), Universit\'e de Poitiers, CNRS, 86073 Poitiers, France}

\author[M. Brachet]{Matthieu Brachet$^\dag$}

\author[G. Rousseaux]{Germain Rousseaux$^\ddag$}
\address{$^\ddag$Team Curiosity, Institut Pprime (CNRS UPR 3346), CNRS, Universit\'e de  Poitiers, ISAE‐ENSMA, 86073 Poitiers, France}

\author[M. Pierre]{Morgan Pierre$^\dag$}


\begin{document}

\begin{abstract}
We derive a linear model of navigation in a two-layer fluid with a variable velocity of the ship. A spectral version of the model including a Rayleigh damping term is analyzed. We prove that the Cauchy problem has a unique solution if the velocity   and if the initial data are  sufficiently regular.  The case of a constant speed is thoroughly investigated and the importance of a critical speed which separates two types of regimes is pointed out. We propose a numerical scheme based on the discrete Fourier transform for the space discretization and on  an exponential integrator for the time discretization. We prove an error estimate for the exponential integrator. Numerical experiments in one and two space dimensions complete the theoretical results. 
\end{abstract}

\maketitle

\section{Introduction}
When a ship navigates in stratified waters, waves propagate at the interface between two layers of fluids with different densities, in addition to the surface wake. These so-called interval waves  may influence the speed of the ship by generating a wave resistance to its movement. This is  known as the dead water phenomenon, and it was observed by Nansen during his expedition of 1893-1896 to the North pole~\cite{nansen1898}. 

The first experiments on dead water were performed by Ekman in 1904~\cite{Ekman1904}. Since then, this phenomenon has drawn a lot of attention in the scientific community. Many authors used linear models to describe it, starting with Lamb~\cite{Lamb1916}. In~\cite{Hudimac1961}, Hudimac derived a formula for the wave resistance of thin ships in a two-layer fluid, by means of a Green function and Fourier transform (see also~\cite{Sretenskii1959}).    In~\cite{Motygin1997}, the thin ship assumption was removed by Motygin and Kuznetsov, for a two-dimensional body. The possibility of non-uniqueness  for this problem for exceptional values of  the speed and geometry of the body was considered recently in~\cite{motygin2023}.

Motivated by the experiments and models in Mercier et al.~\cite{mercier2011}, Duch\^ene analyzed mathematically some nonlinear models in~\cite{duchene2011}. More recently, nonlinear models and numerical simulations were used by Grue to understand more thoroughly the dead water effect   in the case  Nansen's ship (the Fram)~\cite{grue2015,grue2018}. 
Computational Fluid Dynamics (CFD) simulations were also carried out in~\cite{esmaeilpour2018,reid2024}. Nice photographs of internal wave patterns can be found in~\cite{hughes1978,watson1992}.

In Fourdrinoy et al.~\cite{fourdrinoy2020dual}, experiments pointed out two complementary regimes in the dead water effect. First,  a transient regime called Ekman's drag appears; it  is  characterized by oscillations in the speed of the ship and it is related to the initial acceleration and the lateral confinement. Second, an asymptotic regime called Nansen's drag takes place; it is   characterized by a stationary internal wake and a stationary drag. Numerical simulations based on a linear model showed a good agreement with the experiments. This approach was further investigated by Fourdrinoy in his PhD thesis made under the supervision of Rousseaux~\cite{Fourdrinoy2022}. 

\medskip

In this paper, we study a linear model of navigation in two layers of  fluid with  finite depths and different densities. The model is inspired by~\cite{fourdrinoy2020dual}, but instead of imposing of towing force, we assume that the (possibly variable) velocity of the ship is given, as in towing tank experiments which display the wave resistance as a function  of the ship speed. This simplifies the analyzis, but we believe that it is an important step to reach a mathematical understanding of the full model from~\cite{fourdrinoy2020dual}. The fluids are  inviscid and irrotational and the 2D or 3D ship is slender, i.e. its draft is small compared to its length and to its beam. We also add to the model an artificial damping term known as  Rayleigh's trick~\cite{Rayleigh1883}. This is especially useful to handle the steady state solution. 

Starting with the potential flow equations with linear boundary conditions, we apply a Fourier transform and we obtain a spectral formulation of the model.  The resulting linear problem in time involves an unbounded multiplication operator related to  the dispersion relation and a source term depending on the ship's speed (see equation~\eqref{eq:WaterWaves_viscous2}). The unknown of the problem involves the deformation of the interface of the two layers and a potential flow variable. The space variable for the interface is 1D or 2D. 

We show that the Cauchy problem is globally well-posed with or without viscosity. It has a mild solution  if the initial solution belongs to $L^2$ and a strong or classical solution if the initial data is more regular, with appropriate assumptions on the speed. In the case of a constant speed, we give an analytical formulation of the solution. A stationary solution to the problem is provided in the case with viscosity. We recover a critical speed (see~\eqref{eq:Uc}) which is well-known in this context~\cite{duchene2011,fourdrinoy2020dual}. Differences between the subcritical regime and the supercritical regime are pointed out in the 1D and 2D cases. 

We propose a numerical scheme based on the discrete Fourier transform for the space discretization and on an exponential integrator
for the time discretization~\cite{hochbruck1998exponential,hochbruck2010exponential}. The space discretization is natural in view of the spectral formulation of the model. We prove that the error estimate for the time discretization based on a rectangle method has order one.  Moreover, we show that the exponential integrator preserves exactly the dispersion and dissipation properties of the continuous problem. This is a great benefit of the exponential integrator for such wave models~\cite{brachet2022comparison}. Numerical experiments
in one and two space dimensions show the importance of the Rayleigh viscosity for the steady state solution. They also illustrate the subcritical and supercritical regimes for a constant or a variable velocity of the ship.

\medskip

The manuscript is organized as follows. Once the relevant notations have been introduced in Section~\ref{sec:notations}, we derive the model and its spectral formulation in Section~\ref{sec:model}. Then, in Section~\ref{sec:continuouspb}, we analyze mathematically the continuous problem. The space and time discretizations are  briefly presented in Section~\ref{sec:numan}, together with a numerical analysis of the exponential integrator. Extensive 1D and 2D numerical simulations are shown in the last section. In Appendix~\ref{app:exponentialintegrator}, we present an algorithm which finds an optimal regularization parameter for the computation of a steady state solution. 

\section{Notation}
\label{sec:notations}

In the following, $z \in \mathbb{R}$ corresponds to the coordinate in the vertical direction and $\bmx$ is a point of the horizontal plane. For a two-dimensional problem, we have $\bmx = x \in \mathbb{R}$, while for a three-dimensional problem, $\bmx = \begin{pmatrix} x & y \end{pmatrix}^{\top} \in \mathbb{R}^2$. To simplify notations, we denote $\Omega$ the horizontal space, so $\Omega = \mathbb{R}$ or $\Omega = \mathbb{R}^2$.

We will consider different operators: the Laplacian $\Delta$, the divergence $\nabla \cdot $ and the gradient $\nabla$. When these operators are subscripted with $\bmx$ (i.e. $\nabla_{\bmx}$ or $\Delta_{\bmx}$), it means they are considered in the horizontal plane.

Furthermore, we denote $\bmn_s$ as a unit vector orthogonal to the surface $z = s(\bmx)$ and oriented in the $z>0$ direction:
\begin{equation*}
    \bmn_s = \dfrac{1}{\sqrt{1+\normeuclide{ \nabla_{\bmx} s }^2}} \begin{pmatrix} - \nabla_{\bmx} s \\ 1 \end{pmatrix}.
\end{equation*}
For instance, the unit vector orthogonal to the horizontal plane is $\bm{e_z} = \begin{pmatrix} \bm{0} \\ 1 \end{pmatrix}$.

In this paper, we consider the Fourier transform in $\Omega$ defined, for any function $\xi \in L^2(\Omega, \mathbb{C})$ regular enough, by 
\begin{equation}
\label{eq:fourier}
    \hat{\xi}_{\kfourier} = \displaystyle\int_{\Omega} \xi(\bmx)\exp(-2\pi i \kfourier \cdot \bmx) d \bmx.
\end{equation}
where $\normeuclide{ \kfourier } = \left\lbrace \begin{array}{cl}
    |\kappa| & \text{ if } \Omega = \mathbb{R} \text{ and } \kfourier = \kappa \in \mathbb{R} \\
    \sqrt{\kappa_x^2 + \kappa_y^2} & \text{ if } \Omega = \mathbb{R}^2 \text{ and } \kfourier = \begin{pmatrix} \kappa_x & \kappa_y \end{pmatrix}^{\top} \in \mathbb{R}^2
\end{array} \right.$. \\
With this definition, we have the following equalities :
\begin{align*}
    \widehat{\Delta_{\bmx} \xi}_{\kfourier} & = - 4 \pi^2 \normeuclide{ \kfourier }^2 \hat{\xi}_{\kfourier} \\
    \widehat{\nabla_{\bmx} \xi}_{\kfourier} & = 2 \pi i \kfourier \hat{\xi}_{\kfourier} \\
    \widehat{\xi(\cdot - X)}_{\kfourier} & = e^{-2 \pi i \kfourier \cdot X} \hat{\xi}_{\kfourier} \quad \text{ for all } X \in \Omega.
\end{align*}

\section{Model}
\label{sec:model}

\subsection{Model derivation}

Consider a fluid composed of two layers. The top layer is composed of a fluid of density $\rho_1$ and thickness $h_1>0$, while the lower layer has a density $\rho_2$ and thickness $h_2>0$. Densities satisfy $0< \rho_1 < \rho_2$. 

At time $t \geq 0$, a ship, located at $X(t)$, moves with the velocity $\bmU(t) = X'(t)$. As it sails, it generates waves at the interface between the two layers as illustrated in Figure~\ref{fig:twofluidlayer-boat}. In this section, we are interested in modeling these waves represented by a function $\eta : (t,\bmx) \mapsto \eta(t,\bmx) \in \mathbb{R}$. To simplify, the bottom is assumed to be flat and  we assume no free surface variation.  In laboratory experiments,  this latter condition is satisfied if the speed is below 23 $\si{cm/s}$ in water, the so-called Landau speed due to surface tension (see, e.g.,~\cite{carusotto2013}). 
\begin{figure}[ht]
    \centering
    \begin{tikzpicture}[scale=.8]
        \draw [thick, color=black] (-6,0) -- (5,0);
        \fill[color=gray!20] (-6,0) -- (5,0) -- (5,-1) -- (-6,-1);
        \draw [dotted, color=black] (-7.5,0) -- (5,0);
        \draw[color=blue,thick] (-6,5) -- (-1.5,5);
        \draw [blue, thick, domain=-1.5:1.5,smooth] plot (\x, {5+0.5*((\x)^2-2.25)});
        \draw[color=blue,thick] (1.5,5) -- (5,5);
        \draw [->,>=latex,blue] (1.5,4.5) -- (3.5,4.5);
        \draw (2.5,4.5) node[below,blue]{$\bmU(t)$};
        \draw [dotted, blue] (-7.5,5) -- (5,5);
        \draw [dotted, blue] (0,-1) -- (0,5.8);
        \draw (0,6) node[above,blue]{$X(t)$};
        \draw (5,5) node[right,blue]{$s(t,\bmx)$};
        \draw [red, thick] (1,2.5) -- (5,2.5);
        \draw [red, thick, domain=-6:1,smooth,samples=100] plot (\x, {2.5+.6*sin(pi*\x r)*(1-\x)^2*exp(\x)});
        \draw (5,2.5) node[right,red]{$\eta(t,\bmx)$};
        \draw [dotted, red] (-7.5,2.5) -- (5,2.5);
        \draw (-4,1.2) node{$\rho_2$};
        \draw (-4,3.7) node{$\rho_1$};
        \draw [->,>=latex] (-7,-.5) -- (-7,5.5);
        \draw (-7.1,0) node[left] {$z=-h_2$};
        \draw (-7.1,2.5) node[left] {$z=0$};
        \draw (-7.1,5) node[left] {$z=h_1$};
        \draw [->,>=latex] (5,0) -- (5.5,0);
        \draw (5.5,0) node[right]{$\bmx$};
    \end{tikzpicture}
    \caption{Illustration of the dead water phenomenon. $s$ represents the surface on which a boat sails at velocity $\bmU(t)$ over stratified water. $X(t)$ is the boat position at time $t$ and $\eta$ corresponds to the interface between two layers of fluid.}
    \label{fig:twofluidlayer-boat}
\end{figure}
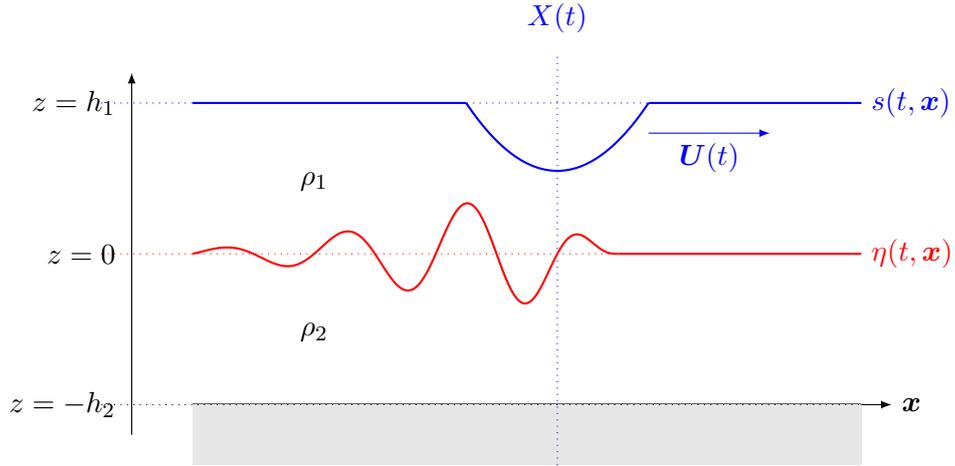

The flows are irrotational, then there exist a potential functions $\Phi_j (t,\bmx,z) \mapsto \Phi_j(t,\bmx,z)$ related to velocity fields by $\bmu_j = \nabla \Phi_j$ (where $\bmu_j : (t,\bmx,z) \mapsto \bmu_j (t,\bmx,z)$ is the velocity of the fluid in the layer $j \in \{ 1,2 \}$). More over, the flow is incompressible then $\nabla \cdot \bmu_j = 0$. From this two properties, we deduce
\begin{align}
    \Delta \Phi_1 = 0 & \qquad t \geq 0, \, \bmx \in \Omega \, \text{ and } \, \eta(t,\bmx) \leq z \leq s(t,\bmx) \label{eq:Poisson1}\\
    \Delta \Phi_2 = 0 & \qquad t \geq 0, \,  \bmx \in \Omega \, \text{ and } \, -h_2 \leq z \leq \eta(t,\bmx). \label{eq:Poisson2}
\end{align}
The model is closed with appropriate boundary conditions corresponding to impermeability and non-mixing of layers conditions (note that mixing was taken into account in~\cite{fourdrinoy2020dual} with Grue's dispersion relation). 
\begin{itemize}
    \item At the bottom $z = -h_2$, the impermeability condition is $\bmu_j \cdot \bm{k} = 0$ with $\bm{e_z}$. Then, we have 
    \begin{equation}
        \dfrac{\partial \Phi_2}{\partial z}(t,\bmx,-h_2) = 0 \qquad \text{ for all } \bmx \in \Omega, \, t \geq 0.
        \label{eq:Bottom}
    \end{equation}

    \item At the surface $z=s(t,\bmx)$, the vertical variation of $s$ is related to the fluid velocity by
    \begin{equation*}
        \dfrac{ds}{dt}(t,\bmx)=\bmu_1(t,\bmx,s(t,z)) \cdot \bmn_s \qquad t \geq 0, \, \bmx \in \Omega \, \text{ and } \, z = s(t,\bmx).
    \end{equation*}
    Assuming the surface is a rigid flat roof with a ship moving, the surface is written such that $s(t,\bmx)=h_1 + f(\bmx-X(t))$ where $f : \bmx \in \Omega \mapsto \mathbb{R}$ corresponds to the boat shape. The previous surface condition at $z = h_1 + f(t,\bmx-X(t))$ becomes
    \begin{align}
        \bmU(t) \cdot \nabla_{\bmx} f(\bmx - X(t)) & = \dfrac{- \nabla_{\bmx} \Phi_1(t,\bmx,z) \cdot \nabla_{\bmx} f(t,\bmx-X(t)) + \tfrac{\partial \Phi_1}{\partial z}(t,\bmx,z)}{\sqrt{1 + \normeuclide{ \nabla_{\bmx} f }^2}}
        \label{eq:Surface}
    \end{align}
    
    \item At the interface $z = \eta(t,\bmx)$, the non-mixing condition is
    \begin{equation*}
        \dfrac{\partial \eta}{\partial t} = \bmu_1 \cdot \bm{n}_{\eta} = \bmu_2 \cdot \bm{n}_{\eta}.
    \end{equation*}
    The interface being set up by $z = \eta(t,\bmx)$, we obtain the relation :
    \begin{equation}
        \nabla_{\bmx}(\Phi_2 - \Phi_1 ) \cdot \nabla_{\bmx} \eta = \dfrac{\partial}{\partial z}(\Phi_2 - \Phi_1 ).
        \label{eq:Interface1}
    \end{equation}
\end{itemize}
Let $g = 9.81 ~\si{m \cdot s^{-2}}$ the standard acceleration of gravity. The Bernoulli's theorem gives the relation
\begin{equation*}
    \rho_j \dfrac{\partial \Phi_j}{\partial t} + \dfrac{1}{2} \normeuclide{ \nabla \Phi_j }^2 + \rho_j gz = C_j
\end{equation*}
where $C_j$ is a constant. Assuming pressure continuity at the interface $z = \eta(t,\bmx)$, we get
\begin{equation}
    \rho_1 \dfrac{\partial \Phi_1}{\partial t} - \rho_2 \dfrac{\partial \Phi_2}{\partial t} + \dfrac{1}{2} \left( \normeuclide{ \nabla \Phi_1 }^2 - \normeuclide{ \nabla \Phi_2 }^2 \right) = g \eta (\rho_2 - \rho_1).
    \label{eq:Interface2}
\end{equation}
The complete non-linear model is given by~\eqref{eq:Poisson1}-\eqref{eq:Interface2}. However, in this paper we focus on the analysis of small variation of $\eta$ and $s$ around a state at rest. We also assume that the ship is slender, i.e. $\normeuclide{ \nabla_{\bmx} f }$ is small. For this reason, we consider the linearized version of the model :
\begin{align}
    \Delta \Phi_1 & = 0 & \qquad t \geq 0, \, \bmx \in \Omega \, \text{ and } \, 0 \leq z \leq h_1 \label{eq:Poisson1_lin}\\
    \Delta \Phi_2 & = 0 & \qquad t \geq 0, \, \bmx \in \Omega \, \text{ and } \, -h_2 \leq z \leq 0 \label{eq:Poisson2_lin}\\
    \dfrac{\partial \Phi_2}{\partial z}(t,\bmx,-h_2) & = 0 & \qquad t \geq 0, \, \bmx \in \Omega, \label{eq:bottom_lin}\\
    -\bmU(t) \cdot \nabla_{\bmx} f(\bmx - X(t)) & = \dfrac{\partial \Phi_1}{\partial z}(t,\bmx,h_1) & \qquad t \geq 0, \, \bmx \in \Omega, \label{eq:surface_lin}\\
    \dfrac{\partial \Phi_1}{\partial z}(t,\bmx,0) & = \dfrac{\partial \Phi_2}{\partial z}(t,\bmx,0)& \qquad t \geq 0, \, \bmx \in \Omega, \label{eq:interface0_lin}\\
    \dfrac{\partial \eta}{\partial t}(t,\bmx) & = \dfrac{\partial \Phi_2}{\partial z}(t,\bmx,0) & \qquad t \geq 0, \, \bmx \in \Omega, \label{eq:interface1_lin}\\
    \dfrac{\partial \varphi}{\partial t}(t,\bmx,0) & = g \eta(t,\bmx,0) (\rho_2 - \rho_1) & \qquad t \geq 0, \, \bmx \in \Omega, \label{eq:interface2_lin}\\
    \varphi(t,\bmx) & = \rho_1 \Phi_1(t,\bmx,0) - \rho_2 \Phi_2(t,\bmx,0) & \qquad t \geq 0, \, \bmx \in \Omega \label{eq:varphi}
\end{align}
The one-dimensional model (namely, with $\bmx=x\in\mathbb{R}$) corresponds to a ship of infinite beam in a channel of infinite width. 

\subsection{Spectral model}
The model~\eqref{eq:Poisson1_lin}-\eqref{eq:interface2_lin} is viewed as a dynamical system in the form
\begin{align}
    \dfrac{\partial \eta}{\partial t} & = T[\varphi(t,\cdot)] \\
    \dfrac{\partial \varphi}{\partial t}(t,\bmx,0) & = g \eta(t,\bmx,0) (\rho_2 - \rho_1)
\end{align}
where the operator $T : \varphi(t,\cdot) \mapsto \tfrac{\partial \Phi_2}{\partial z}(t,\bmx,0)$ and for a given $\varphi(t,\cdot)$, the functions $\Phi_1$ and $\Phi_2$ are defined by equations~\eqref{eq:Poisson1_lin}-\eqref{eq:interface0_lin} and~\eqref{eq:varphi} assuming appropriate boundary conditions at infinity. Nevertheless, the operator $T$ is difficult to analyze in physical space. We will consider it in Fourier space.

\begin{proposition}
    \label{prop:Fourier:T}
    We have the following Fourier relation :
    \begin{equation}
        \widehat{T[\varphi(t,\cdot)]}_{\kfourier} = - \dfrac{2 \pi \normeuclide{ \kfourier }}{T_{1,\kfourier} + T_{2,\kfourier}} \hat{\varphi}_{\kfourier} - \hat{g}_{\kfourier}
    \end{equation}
    where $T_{j,\kfourier} = \rho_j \coth(2 \pi \normeuclide{ \kfourier } h_j)$ ($j \in \{1, 2\}$) and 
    \begin{equation}
    \label{eq:defgk}    
    \hat{g}_{\kfourier} = \dfrac{2 \pi \rho_1 i \kfourier \cdot \bmU(t) e^{-2 \pi i \kfourier \cdot X(t)} \hat{f}_{\kfourier}}{\sinh(2 \pi \normeuclide{ \kfourier } h_1)(T_{1,\kfourier} + T_{2,\kfourier})}.
    \end{equation}
\end{proposition}

\begin{proof}
To compute $\widehat{T[\varphi(t,\cdot)]}_{\kfourier}$, we compute the relation between $\hat{\varphi}_{\kfourier}$ and $\widehat{\tfrac{\partial \Phi_2}{\partial z}(t,\cdot,0)}_{\kfourier}$.\\
From~\eqref{eq:Poisson1_lin} and~\eqref{eq:Poisson2_lin}, for all $j \in \{ 1, 2 \}$ we have $\Delta_{\bmx} \Phi_j + \tfrac{\partial^2 \Phi_j}{\partial z^2} = 0$. The Fourier transform of $\Phi_j$ gives
\begin{equation*}
    \dfrac{\partial^2 \hat{\Phi}_{j,\kfourier}}{\partial z^2} = 4 \pi^2 \normeuclide{ \kfourier }^2 \hat{\Phi}_{j,\kfourier}.
\end{equation*}
Solving this second order ODE gives us the existence of $A_{j,\kfourier}$ and $B_{j,\kfourier}$ such that
\begin{align*}
    \hat{\Phi}_{1,\kfourier}(t,z) & = A_{1,\kfourier} \cosh(2 \pi \normeuclide{ \kfourier }(z-h_1)) +  B_{1,\kfourier} \sinh(2 \pi \normeuclide{ \kfourier }(z-h_1)) \\
    \hat{\Phi}_{2,\kfourier}(t,z) & = A_{2,\kfourier} \cosh(2 \pi \normeuclide{ \kfourier }(z+h_2)) +  B_{2,\kfourier} \sinh(2 \pi \normeuclide{ \kfourier }(z+h_2)).
\end{align*}
We investigate the values $A_{j,\kfourier}$ and $B_{j,\kfourier}$ by considering the boundary conditions of each fluid layer.\\
First, the bottom condition~\eqref{eq:bottom_lin} gives $B_{2,\kfourier} = 0$ so
\begin{equation}
    \label{eq:proof:PHI2}
    \hat{\Phi}_{2,\kfourier}(t,z) = A_{2,\kfourier} \cosh(2 \pi \normeuclide{ \kfourier }(z+h_2)).
\end{equation}
More over, the Fourier transform of~\eqref{eq:surface_lin} gives $\tfrac{\partial \hat{\Phi}_{1,\kfourier}}{\partial z}(t,h_1) = - 2 \pi i \kfourier \cdot \bmU(t) e^{-2\pi i \kfourier \cdot X(t)} \hat{f}_{\kfourier}$ so we deduce 
\begin{equation}
    \label{eq:proof:B1}
    B_{1,\kfourier} = - \tfrac{i \kfourier \cdot \bmU(t) e^{-2\pi i \kfourier \cdot X(t)}}{\normeuclide{ \kfourier }}\hat{f}_{\kfourier}.
\end{equation}
Equation~\eqref{eq:interface1_lin} gives $\tfrac{\partial \Phi_1}{\partial z}(t,\bmx,0) = \tfrac{\partial \Phi_2}{\partial z}(t,\bmx,0)$ then
\begin{equation}
    \label{eq:proof:syst1}
    A_{1,\kfourier} \sinh(2 \pi \normeuclide{ \kfourier } h_1) + A_{2,\kfourier} \sinh(2 \pi \normeuclide{ \kfourier } h_2) = B_{1,\kfourier} \cosh ( 2 \pi \normeuclide{ \kfourier } h_1 ).
\end{equation}
This equality is the first equation of a system. The second equality follows from~\eqref{eq:interface2_lin}:
\begin{equation}
    \label{eq:proof:syst2}
    A_{1,\kfourier} \rho_1 \cosh(2 \pi \normeuclide{ \kfourier } h_1) - A_{2,\kfourier} \rho_2 \cosh(2 \pi \normeuclide{ \kfourier } h_2) = \hat{\varphi}_{\kfourier}(t) + \rho_1 B_{1,\kfourier} \sinh(2 \pi \normeuclide{ \kfourier } h_1).
\end{equation}
From~\eqref{eq:proof:syst1}-\eqref{eq:proof:syst2}, we deduce that 
\begin{equation}
    \label{eq:proof:A2}
    A_{2,\kfourier} = - \dfrac{\hat{\varphi}_{\kfourier}}{\sinh(2 \pi \normeuclide{ \kfourier } h_1) \sin(2 \pi \normeuclide{ \kfourier } h_2) (T_{1,\kfourier}+T_{2,\kfourier})} + \dfrac{\rho_1 B_{1,\kfourier}}{\sinh(2 \pi \normeuclide{ \kfourier } h_1) \sin(2 \pi \normeuclide{ \kfourier } h_2) (T_{1,\kfourier}+T_{2,\kfourier})}.
\end{equation}
To end the proof, we have 
\begin{align*}
    \widehat{T[\varphi(t,\cdot)]}_{\kfourier} & = \dfrac{\partial \hat{\Phi}_{2,\kfourier}}{\partial z}(t,0) & \text{ by definition,}\\
        & = 2 \pi \normeuclide{ \kfourier } A_{2,\kfourier} \sinh(2 \pi \normeuclide{ \kfourier } h_2 ) & \text{ using~\eqref{eq:proof:PHI2},} \\
        & = - \dfrac{2 \pi \normeuclide{ \kfourier }}{T_{1,\kfourier}+T_{2,\kfourier}} + \dfrac{2 \pi \rho_1 \normeuclide{ \kfourier } B_{1,\kfourier}}{\sinh(2 \pi \normeuclide{ \kfourier } h_1) (T_{1,\kfourier}+T_{2,\kfourier})} & \text{ with~\eqref{eq:proof:A2},} \\
        & = - \dfrac{2 \pi \normeuclide{ \kfourier }}{T_{1,\kfourier} + T_{2,\kfourier}} \hat{\varphi}_{\kfourier} - \hat{g}_{\kfourier}  & \text{ thanks to~\eqref{eq:proof:B1}}.
\end{align*}
This last relation concludes the proof.
\end{proof}
Proposition~\ref{prop:Fourier:T} and equation~\eqref{eq:interface2_lin} give the following linear dynamical system:
\begin{align*}
    \dfrac{d \hat{\eta}_{\kfourier}}{d t} &  = - \dfrac{2 \pi \normeuclide{ \kfourier }}{T_{1,\kfourier} + T_{2,\kfourier}} \hat{\varphi}_{\kfourier} - \hat{g}_{\kfourier} \\
    \dfrac{d \hat{\varphi}_{\kfourier}}{d t} & = g (\rho_2 - \rho_1) \hat{\eta}_{\kfourier}.
\end{align*}
Diagonalizing the matrix $\begin{bmatrix}
     0 & - \tfrac{2 \pi \normeuclide{ \kfourier }}{T_{1,\kfourier} + T_{2,\kfourier}}\\
     g (\rho_2 - \rho_1) & 0
    \end{bmatrix}$ leads to two decoupled equations. We are going to focus on one of them, which allows us to compute the interface $\eta$ over time.
\begin{proposition}
\label{prop:Waterwaves}
Let $\mudiff_{\kfourier} = \hat{\eta}_{\kfourier} + i \sqrt{\dfrac{2 \pi \normeuclide{ \kfourier }}{(T_{1,\kfourier}+T_{2,\kfourier})(\rho_2-\rho_1)g}} \hat{\varphi}_{\kfourier}$. Then, function $\mudiff_{\kfourier} : t \geq 0 \mapsto \mudiff_{\kfourier}(t) \in \mathbb{C}$ satisfies
\begin{equation}
    \label{eq:WaterWaves}
    \dfrac{d \mudiff_{\kfourier}}{dt} = i \omega_{\kfourier} \mudiff_{\kfourier} - \hat{g}_{\kfourier}
\end{equation}
with the notation introduced in Proposition~\ref{prop:Fourier:T} and
\begin{equation}
    \label{eq:omega_k}
\omega_{\kfourier} = \sqrt{\dfrac{2 \pi \normeuclide{ \kfourier } (\rho_2-\rho_1)g}{T_{1,\kfourier}+T_{2,\kfourier}}}
\end{equation}
\end{proposition}

\begin{proof}
Equation~\eqref{eq:WaterWaves} is obtained with simple calculation from $\mudiff_{\kfourier}$.
\end{proof}

Equation~\eqref{eq:omega_k} is known as the  \textit{dispersion relation} for this problem. It was first introduced by Stokes in 1847 for two layers of fluids with different densities~\cite{Stokes}. In equation~\eqref{eq:WaterWaves}, the interaction with the ship is entirely contained in the term $\hat{g}_{\kfourier}$.

As it will be observed in the rest of this work and in order to guarantee the existence of a solution in $L^2(\Omega)$ in every situation, we consider the equation in the ``viscous'' form :
\begin{align*}
    \dfrac{d \hat{\eta}_{\kfourier,\varepsilon}}{d t} &  = - \dfrac{2 \pi \normeuclide{ \kfourier }}{T_{1,\kfourier} + T_{2,\kfourier}} \hat{\varphi}_{\kfourier,\varepsilon} - \hat{g}_{\kfourier} - \varepsilon \hat{\eta}_{\kfourier,\varepsilon}\\
    \dfrac{d \hat{\varphi}_{\kfourier,\varepsilon}}{d t} & = g (\rho_2 - \rho_1) \hat{\eta}_{\kfourier,\varepsilon} - \varepsilon \hat{\varphi}_{\kfourier,\varepsilon}.
\end{align*}
where $\varepsilon>0$ is a small damping parameter.  Adding this transforms equation~\eqref{eq:WaterWaves} into
\begin{equation}
    \label{eq:WaterWaves_viscous}
    \dfrac{d \mudiff_{\kfourier,\varepsilon}}{dt} = (i \omega_{\kfourier} - \varepsilon) \mudiff_{\kfourier,\varepsilon} - \hat{g}_{\kfourier}
\end{equation}
This trick is well-known in the context of wave resistance (see, e.g.,~\cite{Havelock}), and the additional term is known as a Rayleigh artificial   damping term~\cite{Rayleigh1883} (see also~\cite[p. 399]{Lamb1924}). 

Starting with the same initial condition $\mudiff_{\kfourier,\varepsilon}(0) = \mudiff_{\kfourier}(0)$, we expect to have $\mudiff_{\kfourier,\varepsilon} \rightarrow \mudiff_{\kfourier}$ when $\varepsilon$ goes to $0$. Thus, simulations will be realized with $\varepsilon$ as small as possible. 

\section{The continuous problem}
\label{sec:continuouspb}
In this section, we analyze equation~\eqref{eq:WaterWaves_viscous}, which can be written
\begin{equation}
    \label{eq:WaterWaves_viscous2}
    \dfrac{d \mudiff_{\kfourier,\varepsilon}}{dt}(t) = i \omega_{\kfourier,\varepsilon}  \mudiff_{\kfourier,\varepsilon}(t) - \hat{g}_{\kfourier}(t),
\end{equation}
where $\omega_{\kfourier,\varepsilon}=\omega_{\kfourier}+i\varepsilon$ with $\varepsilon\ge 0$, $\omega_{\kfourier}$ is defined by~\eqref{eq:omega_k} and $\hat{g}_{\kfourier}(t)$ is defined by~\eqref{eq:defgk}. In $\hat{g}_{\kfourier}(t)$, the speed $U(t)$ is given and the position of the ship is $X(t)=\int_0^tU(s)ds$.

By Duhamel's formula, the solution to~\eqref{eq:WaterWaves_viscous2} reads (at least formally)
\begin{equation}
    \label{eq:Duhamel}
    \mudiff_{\kfourier,\varepsilon}(t)=\mudiff_{\kfourier,\varepsilon}(0)e^{i\omega_{\kfourier,\varepsilon}t}-\int_0^te^{i\omega_{\kfourier,\varepsilon}(t-\tau)}\hat{g}_{\kfourier}(\tau)d\tau.
\end{equation}
We introduce a functional framework for this formula. 
\subsection{Definition of a solution}
Recall that the space variable $\bmx$ belongs to $\mathbb{R}^d$ with $d=1$ or $2$. For each $t\ge 0$, we consider the function $\mudiff_{\varepsilon}(t):\kfourier\mapsto \mudiff_{\kfourier,\varepsilon}(t)$ as a function in $L^2(\mathbb{R}^d;\mathbb{C})=L^2(\mathbb{R}^d)$ and its $L^2$-norm is 
$$\normLdeux{\mudiff_{\varepsilon}(t)}=\left(\int_{\mathbb{R}^d}\left|\mudiff_{\kfourier,\varepsilon}(t)\right|^2d\kfourier\right)^{\sfrac{1}{2}}.$$
The $L^2$-norm is defined similarly for any function in $L^2(\mathbb{R}^d)$.

In~\eqref{eq:WaterWaves_viscous2}, the multiplication operator $i\omega_{\kfourier,\varepsilon}$ is unbounded in $L^2(\mathbb{R}^d)$ because $\left|\omega_{\kfourier,\varepsilon}\right|=\left(\omega_{\kfourier}^2+\varepsilon^2\right)^{\sfrac{1}{2}}$ behaves like $C\normeuclide{\kfourier}^{\sfrac{1}{2}}$ at $+\infty$, for a positive constant $C$. However, for all $t\ge 0$, $e^{i\omega_{\kfourier,\varepsilon}t}$ is a bounded operator in $L^2(\mathbb{R}^d)$ since $|e^{i\omega_{\kfourier,\varepsilon}t}|=e^{-\varepsilon t}\le 1$. Moreover,  the term $\hat{g}(t):\kfourier\mapsto\hat{g}_{\kfourier}(t)$ behaves well. Indeed, we have:

\begin{lemma}
\label{lem:gk}
Let   $T>0$.  Assume that $\bmU\in L^p(0,T;\mathbb{R}^d)$ with $1\le p\le +\infty$ and $f\in L^2(\mathbb{R}^d)$. Then $\hat{g}$ belongs to $L^p(0,T;L^2(\mathbb{R}^d))$.
\end{lemma}
Here and thereafter, for a given Hilbert space $H$ with norm $\|\cdot\|_H$, we denote by $L^p(0,T;H)$ the usual Banach space of (classes of) functions $v$ from $[0,T]$ into $H$ such that   $t\mapsto\|v(t)\|_H$ is a $p$-th integrable function on $[0,T]$.  Moreover, $C^0([0,T],H)$ is the Banach space of continuous functions from $[0,T]$ into $H$.
\begin{proof}
For almost all $t\in [0,T]$ and for almost all $\kfourier\in\mathbb{R}^d$, we have
\begin{equation}
    \label{eq:gk1}
    \left|\hat{g}_{\kfourier}(t)\right|\le N(\normeuclide{\kfourier})\normeuclide{\bmU(t)}\left|\hat{f}_{\kfourier}\right|,
    \end{equation}
where
\begin{equation}
    \label{eq:Ns}
    N(\normeuclide{\kfourier})=\frac{2\pi\rho_1\normeuclide{\kfourier}}{\sinh(2\pi\normeuclide{\kfourier}h_1)(T_{1,\kfourier}+T_{2,\kfourier})}.
    \end{equation}
The function $s\mapsto N(s)$ is continuous on $(0,+\infty)$. As $s\longrightarrow +\infty$, we have
$$N(s)\sim \frac{4\pi\rho_1 s}{e^{2\pi s h_1}(\rho_1+\rho_2)},$$
so that $N(s)\longrightarrow 0$ as $s\longrightarrow +\infty$.
Similarly, as $s\longrightarrow 0^+$, we have 
$$N(s)\sim\frac{2\pi\rho_1 s}{\rho_1+\rho_2 h_1/h_2},$$
so that $N(s)\longrightarrow 0$ as $s\longrightarrow 0^+$.
Thus, $N$ is bounded on $(0,+\infty)$, that is
$$0\le N(s)\le C(\rho_1,\rho_2,h_1,h_2)\quad\forall s\in(0,+\infty).$$
Consequently, for almost every $t\in[0,T]$, we have 
$$\normLdeux{\hat{g}(t)}\le C(\rho_1,\rho_2,h_1,h_2)\normeuclide{\bmU(t)}\normLdeux{\hat{f}}.$$
By Parseval's theorem, $\normLdeux{f}=\normLdeux{\hat{f}}$. Since  $\bmU\in L^p(0,T)$, we conclude that $\hat{g}$ belongs to $L^p$ as claimed.
\end{proof}
The (mild) solution is defined as follows (see, e.g.,~\cite{EN2000,Trelat2024}). 
\begin{definition}
Let $\varepsilon\ge 0$ and $T>0$. Assume that $\bmU\in L^1(0,T;\mathbb{R}^d)$ and $f\in L^2(\mathbb{R}^d)$. Then for every $\mudiff_{\varepsilon,0}\in L^2(\mathbb{R}^d)$, the function $\mudiff_{\varepsilon}\in C^0\left([0,T];L^2(\mathbb{R}^d)\right)$ defined by Duhamel's formula~\eqref{eq:Duhamel} is called \textbf{the mild solution} of~\eqref{eq:WaterWaves_viscous2} associated to the initial condition $\mudiff_{\varepsilon}(0)=\mudiff_{\varepsilon,0}$.
\end{definition}
\begin{remark}This well-posedness result includes the case $\varepsilon=0$, i.e. the model without damping. The same  is true with more regularity, see Theorems~\ref{th:strong} and~\ref{thm:wellposed}. The Rayleigh damping is especially useful for the stationary solution defined in Proposition~\ref{prop:stat}. It avoids for instance wave radiation in front of the ship.
\end{remark}

\begin{remark}
Given that 
\begin{equation}
    \label{eq:mu_epsilon}
    \mudiff_{\kfourier,\varepsilon}=\hat{\eta}_{\kfourier,\varepsilon}+i\alpha_{\kfourier}\hat{\varphi}_{\kfourier,\varepsilon},
\end{equation}
where $\alpha_{\kfourier} =  \sqrt{\dfrac{ 2 \pi \normeuclide{\kfourier} }{ ( T_{1,\kfourier}+T_{2,\kfourier})(\rho_2-\rho_1)g }}$, and noting that $\eta_\varepsilon$ and $\varphi_\varepsilon$ are real-valued functions, they can be recovered as follows:
\begin{equation}
    \label{eq:eta_varphi}
    \hat{\eta}_{\kfourier,\varepsilon}=\dfrac{\mudiff_{\kfourier,\varepsilon}+\overline{\mudiff_{-\kfourier,\varepsilon}}}{2}\quad\mbox{and}\quad
\hat{\varphi}_{\kfourier,\varepsilon}=\dfrac{\mudiff_{\kfourier,\varepsilon}-\overline{\mudiff_{-\kfourier,\varepsilon}}}{2i\alpha_{\kfourier}}.
\end{equation}
Indeed, since $\eta_\varepsilon$ and $\varphi_\varepsilon$ are real-valued functions, we have $\hat{\eta}_{\kfourier,\varepsilon}=\overline{\hat{\eta}_{-\kfourier,\varepsilon}}$ and $\hat{\varphi}_{\kfourier,\varepsilon}=\overline{\hat{\varphi}_{-\kfourier,\varepsilon}}$, so that 
\begin{equation}
    \label{eq:mu_epsilon2}
    \overline{\mudiff_{-\kfourier,\varepsilon}}=\hat{\eta}_{\kfourier,\varepsilon}-i\alpha_{\kfourier}\hat{\varphi}_{\kfourier,\varepsilon},
\end{equation}
where we used that $\overline{\alpha_{-\kfourier}}=\alpha_{\kfourier}$.
The relations~\eqref{eq:eta_varphi} follow from~\eqref{eq:mu_epsilon} and~\eqref{eq:mu_epsilon2}. The functions $\eta_{\varepsilon}$ and $\varphi_\varepsilon$ are the inverse Fourier transform of $\hat{\eta}_{\kfourier,\varepsilon}$ and $\hat{\varphi}_{\kfourier,\varepsilon}$.
\end{remark}

\begin{remark}
By applying the triangle inequality, we deduce from~\eqref{eq:mu_epsilon} and~\eqref{eq:eta_varphi} that for all $t\ge 0$,
$$\mudiff_{\kfourier,\varepsilon}(t)\in L^2(\mathbb{R}^d)\iff \hat{\eta}_{\kfourier,\varepsilon}(t)\in L^2(\mathbb{R}^d)\quad\text{and}\quad \alpha_{\kfourier}\hat{\varphi}_{\kfourier,\varepsilon}(t)\in L^2(\mathbb{R}^d).$$
\end{remark}

Assuming more regularity on the initial value, it is possible to define a notion of strong solution for problem~\eqref{eq:WaterWaves_viscous2}. This is shown in the next section.

\subsection{Existence and uniqueness of a strong solution}
For each $\varepsilon\ge 0$, we define the unbounded multiplication operator $A_\varepsilon b_{\kfourier}=i\omega_{\kfourier,\varepsilon}b_{\kfourier}$ with domain
\begin{eqnarray*}
D(A_\varepsilon)&=&\left\{b\in L^2(\mathbb{R}^d)\ |\ A_\varepsilon b\in L^2(\mathbb{R}^d)\right\}\\
&=&\left\{b\in L^2(\mathbb{R}^d)\ |\ \int_{\mathbb{R}^d}\left(\omega_{\kfourier}^2+\varepsilon^2\right)|b_{\kfourier}|^2d\kfourier<+\infty\right\}
\end{eqnarray*}
The operator $A_\varepsilon$ is densily defined, closed and m-accretive in $L^2(\mathbb{R}^d;\mathbb{C})$. It generates the semigroup of contractions $t\mapsto e^{i\omega_{\kfourier,\varepsilon}t}$ in $L^2(\mathbb{R}^d)$. For each $\varepsilon\ge 0$, $D(A_\varepsilon)$ is endowed with the Hilbertian norm
$$b\mapsto\left(\int_{\mathbb{R}^d}\left(1+\omega_{\kfourier}^2+\varepsilon^2\right)|b_{\kfourier}|^2d\kfourier\right)^{\sfrac{1}{2}}.$$
Moreover, $D(A_\varepsilon)$ is a Hilbert space for this norm.

Equation~\eqref{eq:WaterWaves_viscous2} is a non-homogenous problem involving $A_\varepsilon$. Concerning the source term $\hat{g}$, we have:
\begin{lemma}
Let $\varepsilon\ge 0$, $T>0$ and assume that  $f\in L^2(\mathbb{R}^d)$. If  $\bmU\in L^p(0,T;\mathbb{R}^d)$ with $1\le p\le +\infty$, then $\hat{g}$ belongs to $L^p(0,T;D(A_\varepsilon))$. Moreover, if $\bmU\in C^0([0,T];\mathbb{R}^d)$, then $\hat{g}$ belongs to $C^0([0,T];D(A_\varepsilon))$. 
\end{lemma}
\begin{proof}By~\eqref{eq:gk1}, for almost all $t\in[0,T]$ and almost all $\kfourier\in\mathbb{R}^d$, we have
\begin{equation}
    \label{eq:gk2}
    \left|i\omega_{\kfourier,\varepsilon}\hat{g}_{\kfourier}(t)\right|\le |\omega_{\kfourier}^2+\varepsilon^2|^{\sfrac{1}{2}} N(\normeuclide{\kfourier})\normeuclide{\bmU(t)}\left|\hat{f}_{\kfourier}\right|,
    \end{equation}
    where $N$ is defined by~\eqref{eq:Ns}. By arguing as in the proof of Lemma~\ref{lem:gk}, we see that the function 
    $$s\mapsto |\omega_{s}^2+\varepsilon^2|^{\sfrac{1}{2}} N(s)$$
    is continuous on $(0,+\infty)$, and that it tends to $0$ as $s$ tends to $0$ or $+\infty$. Indeed, $\omega_{\kfourier}$ is equivalent to a constant times $\normeuclide{\kfourier}^{\sfrac{1}{2}}$ near $+\infty$, and $\omega_{\kfourier}$ is equivalent to a constant times $\normeuclide{\kfourier}$ near $0$. Thus, there is a constant $\tilde{C}$ (depening on $\rho_1$, $\rho_2$, $h_1$ and $h_2$) such that 
    $$\normLdeux{A_\varepsilon\hat{g}(t)}\le \tilde{C}\normeuclide{\bmU(t)}\normLdeux{\hat{f}},$$
    for almost all $t\in[0,T]$. If $\bmU$ belongs to $L^p(0,T)$, then $A_\varepsilon \hat{g}$ belongs to $L^p(0,T;L^2(\mathbb{R}^d))$ as claimed. If $\bmU$ is continuous on $[0,T]$, we deduce from Lebesgue's dominated convergence theorem that $A_\varepsilon \hat{g}$ is continuous on $[0,T]$ with values in $L^2(\mathbb{R}^d)$. This completes the proof.
\end{proof}
Using the regularity of $\hat{g}$, we deduce the existence and uniqueness of a {\it strong} solution~\cite{EN2000,Trelat2024}.
\begin{theorem}
\label{th:strong}
Let $\varepsilon\ge 0$ and $T>0$. Assume that  $f\in L^2(\mathbb{R}^d)$ and  $\bmU\in L^p(0,T;\mathbb{R}^d)$ with $1\le p\le +\infty$. Then for every $\mudiff_{\varepsilon,0}\in D(A_\varepsilon)$, there is a unique solution $$\mudiff_\varepsilon\in C^0([0,T];D(A_\varepsilon))\cap W^{1,p}(0,T;L^2(\Rd))$$
of~\eqref{eq:WaterWaves_viscous2} such that $\mudiff_\varepsilon(0)=\mu_{\varepsilon,0}$. Moreover, $\mudiff_\varepsilon$ is given by Duhamel's formula~\eqref{eq:Duhamel}.
\end{theorem}
If $\bmU$ is continuous, we have existence and uniqueness of a {\it classical} solution~\cite{EN2000}. 
\begin{theorem}
\label{thm:wellposed}
Let $\varepsilon\ge 0$ and $T>0$. Assume that  $f\in L^2(\mathbb{R}^d)$ and $\bmU\in C^0([0,T];\Rd)$. Then for every $\mudiff_{\varepsilon,0}\in D(A_\varepsilon)$, there is a unique solution 
$$\mudiff_\varepsilon\in C^0([0,T];D(A_\varepsilon))\cap C^1([0,T];L^2(\Rd))$$
of~\eqref{eq:WaterWaves_viscous2} such that $\mudiff_\varepsilon(0)=\mu_{\varepsilon,0}$. Moreover, $\mudiff_\varepsilon$ is given by Duhamel's formula~\eqref{eq:Duhamel}.
\end{theorem}

\subsection{Solution for a constant speed}
\label{sec:const_U}
In this section, we assume that the ship moves at constant speed $\bmU(t)=\Ux \bm{e_x}$ along the $x-$axis, where $\bm{e_x}$ is a unit vector in the direction of the $x-$axis and $\Ux\in\mathbb{R}$ is constant. 
We assume for simplicity that $X(0)=0$, so that $X(t)=\Ux \, t \, \bm{e_x}$. We first have:
\begin{proposition}
\label{prop:stat}
Let $\varepsilon>0$. The stationary solution to~\eqref{eq:WaterWaves_viscous2}  is given by
\begin{equation}
    \label{eq:mu_stationary}
    \mudiff_{\kfourier,\varepsilon}^\star(t)=\frac{D_{\kfourier}}{i(2\pi\kUx+\omega_{\kfourier,\varepsilon})}e^{-2\pi i\kUx t},
\end{equation}
where 
$$D_{\kfourier}=\frac{2\pi \rho_1\kUx \hat{f}_{\kfourier}}{\sinh(2\pi\normeuclide{\kfourier}h_1)(T_{1,\kfourier}+T_{2,\kfourier})}.$$
\end{proposition}
This stationary solution corresponds to a wave moving at the same speed as the ship and which is a steady state in the frame of reference of the ship. Indeed, if $\mu_{\varepsilon}^\star(t)$ denotes the inverse Fourier transform of $\mudiff_{\varepsilon}^\star(t)$ with initial condition $\mu_{\varepsilon}(0)=\mu_{\varepsilon,0}^\star$, we have
\begin{equation}
    \label{eq:initial_steady}
    \mu_{\varepsilon}^\star(t)=\mu_{\varepsilon,0}^\star(\bm{x}-\Ux t\bm{e_x}),
\end{equation}
by standard operations on the Fourier transform (see~\eqref{eq:fourier}). 

We note that for $\varepsilon=0$, $\mudiff_{0}^\star(t)$ does not generally belong to $L^2(\mathbb{R}^d)$, because the term $2\pi\kUx+\omega_{\kfourier,\varepsilon}$ may vanish for some values of $\kfourier\not=0$ (see Sections~\ref{subsec:1D_theory} and~\ref{subsec:2D_theory}). 
In contrast, for $\varepsilon>0$, this term never vanishes for $\kfourier\not=0$ and  $\mudiff_{\varepsilon}^\star(t)$ belongs to $L^2(\mathbb{R}^d)$.
The parameter $\varepsilon>0$ acts as a regularization of equation~\eqref{eq:WaterWaves}, known as Rayleigh's trick (see, e.g.,~\cite[p. 399]{Lamb1924}).

\begin{proof} We have 
\begin{equation}
    \label{eq:gk}
    \hat{g}_{\kfourier}(t)=D_{\kfourier}e^{-2\pi i\kUx t}.
    \end{equation}
A simple calculation shows that $\mudiff_{\kfourier,\varepsilon}^\star(t)$ given by~\eqref{eq:mu_stationary} solves~\eqref{eq:Duhamel}.
\end{proof}

In the case of a constant speed, the solution can be computed. We have:
\begin{proposition}
\label{prop:solUconst}
Let $\varepsilon\ge 0$ and $\mudiff_{\varepsilon,0}\in L^2(\mathbb{R}^d)$. The solution to~\eqref{eq:WaterWaves_viscous2}  is given by
$$\mu_{\kfourier,\varepsilon}(t)=\mu_{\kfourier,\varepsilon}(0)e^{i\omega_{\kfourier,\varepsilon}t}-\frac{D_{\kfourier}}{i(2\pi\kUx+\omega_{\kfourier,\varepsilon})}e^{i\omega_{\kfourier,\varepsilon}t}\left(1-e^{-i(2\pi \kUx+\omega_{\kfourier,\varepsilon})t}\right).$$
\end{proposition}
\begin{proof}
Since $\hat{g}_k$ is given by~\eqref{eq:gk}, Duhamel's formula~\eqref{eq:Duhamel} yields the solution by integrating with respect to time. 
\end{proof}
In Proposition~\ref{prop:solUconst}, the case $\varepsilon=0$ is also valid since the last term in parentheses cancels the singularity of the denominator. 

\subsection{One-dimensional wave propagation}
\label{subsec:1D_theory}
In this section, we assume that the speed of the ship is constant, i.e. $\bmU(t)=\Ux\bm{e_x}$, as in section~\ref{sec:const_U}. We also assume that the waves propagate in one dimension, i.e. $\bm{x}\in\mathbb{R}$ and $\kfourier=\kappa_x\in\mathbb{R}$. We focus on the case $\varepsilon=0$, so that $\omega_{\kfourier,\varepsilon}=\omega_{\kfourier}$, as defined in~\eqref{eq:omega_k}.  

It is interesting to introduce a critical velocity for the 1D problem, defined by 
\begin{equation}
    \label{eq:Uc}
U_c=\sqrt{\dfrac{(\rho_2-\rho_1)g}{\rho_1/h_1+\rho_2/h_2}}.
\end{equation}
Note that for a water-air interface, where $h_1\longrightarrow +\infty$, $h_2$ is finite and $\rho_2>\!\!>\rho_1$, we have $U_c\approx \sqrt{gh_2}$, a well-known formula~\cite{Lamb1924}.\\

If $\Ux<U_c$, the speed is \textit{subcritical} and if $\Ux>U_c$, the speed is \textit{supercritical}. Let 
\begin{equation}
    \label{eq:phasevelocity}
    v_p(\kfourier)=\frac{\omega_{\kfourier}}{2\pi\normeuclide{\kfourier}}
\end{equation}
denote the phase velocity~\cite{whitham2011linear}. 
We have: 
\begin{theorem}
\label{thm:oneD_wave}
If $\Ux<U_c$, there exists a unique $\kfourier_c^\star>0$ such that $v_p(\kfourier_c^\star)=\Ux$. \\
If $\Ux>U_c$, there is no $k_c>0$ such that $v_p(k_c)=\Ux$. 
\end{theorem}
\begin{proof}Consider the function $P(s)$ defined  by
$$P(s)=\frac{\omega_s}{2\pi s}=\sqrt{\frac{(\rho_2-\rho_1)g}{2\pi s(T_{1,s}+T_{2,s})}},$$
where $s=\normeuclide{\kfourier}=|\kappa_x|>0$. We have 
$$P(s)=\sqrt{\frac{(\rho_2-\rho_1)g}{R(s)}},$$
where $R(s)=2\pi s(\rho_1\coth(2\pi s h_1)+\rho_2\coth(2\pi s h_2))$. The function $P(s)$ is continuous on $(0,+\infty)$ since $R$ is continuous and positive on $(0,+\infty)$. We claim that $R$ is stricly increasing on $(0,+\infty)$, so that $P$ is strictly decreasing on $(0,+\infty)$. To prove this claim, we note that $$R(s)=\frac{\rho_1}{h_1}E(2\pi s h_1)+\frac{\rho_2}{h_2} E(2\pi s h_2),$$
where $E(\sigma)=\sigma \coth(\sigma)$. The function $E$ is strictly increasing on $(0,+\infty)$ because its derivative reads
$$E'(\sigma)=\frac{\sinh(\sigma)\cosh(\sigma)-\sigma}{\sinh^2(\sigma)},$$
so that  $E'(\sigma)>0$ since $\sinh(\sigma)>\sigma$ and $\cosh(\sigma)> 1$   on $(0,+\infty)$.  Thus, $R$ is the sum of two strictly increasing functions and the claim is proved. 

Finally, we note that as $s\longrightarrow 0^+$, $R(s)$ tends to $\rho_1/h_1+\rho_2/h_2$, so that $P(s)\longrightarrow U_c$. As $s\longrightarrow +\infty$, we have $R(s)\longrightarrow +\infty$, so that $P(s)\longrightarrow 0$. Thus, $P$ is a bijection from $(0,+\infty)$ onto $(0,U_c)$ and this completes the proof. 
\end{proof}
Let us point out two consequences of Theorem~\ref{thm:oneD_wave} when $\varepsilon=0$. In the supercritical case, the stationary solution $\mudiff_{\kfourier,0}^\star$ defined by~\eqref{eq:mu_stationary} belongs to $L^2(\mathbb{R}^d)$ because the denominator $2\pi\kappa_x\Ux+\omega_{\kfourier,0}$ does not vanish for any $\kappa_x\not=0$. Consequently, there is no wake behind the ship in the stationary case. 

In contrast, in the subcritical regime, the denominator of $\mudiff_{\kfourier,0}^\star$ is zero for exactly one value of $\kappa_x^\star<0$ (note that $v_p(\kappa_x^\star)=v_p(-\kappa_x^\star)=\Ux$). Consequently, the stationary solution $\mudiff_{\kfourier,0}^\star$ does not belong to $L^2(\mathbb{R}^d)$ for a generic ship-function $f$.  The regularization parameter $\varepsilon>0$ is then needed to approximate the solution.  A 1D sinusoidal wake corresponding to the critical  wave number $\kappa_x^\star$ is observed behind the ship (see, e.g.~\cite{prihandono2021linear}). 

\subsection{Two-dimensional wave propagation}
\label{subsec:2D_theory}
In this section, we assume again that the speed of the ship is constant, i.e. $\bmU(t)=\Ux\bm{e_x}$, as in section~\ref{sec:const_U}. We  assume now that the waves propagate in two dimensions, i.e. $\bm{x}\in\mathbb{R}^2$ and $\kfourier=(\kappa_x,\kappa_y)\in\mathbb{R}^2$. We now set $\varepsilon=0$, hence $\omega_{\kfourier,\varepsilon}=\omega_{\kfourier}$, as defined in~\eqref{eq:omega_k}.

Let us consider the domain of simple waves traveling at the speed of the ship :
\begin{equation*}
    \mathcal{D} = \left\lbrace \kfourier = (\kappa_x, \kappa_y) \in \mathbb{R}^2 \text{ such that } 2 \pi \kappa_x U_x = \omega_{\kfourier} \text{ and } \kfourier \neq (0,0) \right\rbrace.
\end{equation*}
Let $\kfourier \in \mathcal{D}$ and a simple wave given by
\begin{align*}
    \exp {\left(2 i \pi \kfourier \cdot \bmx - i \omega_{\kfourier} t)\right)} & = \exp{\left(2 i \pi \kappa_x (x- U_x t) + 2 i \pi \kappa_y y\right)} \; \text{ as $\kfourier = (\kappa_x, \kappa_y) \in \mathcal{D}$} \\
        & = \exp{\left( 2 i \pi (\kappa_x x' + \kappa_y y') \right)}
\end{align*}
where $(x',y')=(x-U_x t, y)$ corresponds to the reference frame of the ship.
So, the wave follows the ship and is constant along a straight line given by $\kappa_x x' + \kappa_y y' = C$ (for all $C \in \mathbb{R}$ constant) and orming an angle $\varphi$ with the $\bm{e}_x-$axis. We complete the angle by symmetry and obtain
\begin{equation*}
    \varphi = \pm  \left\lbrace  
    \begin{array}{cc}
        \arctan \left( \tfrac{\kappa_x}{\kappa_y} \right) & \text{ if } \kappa_y \neq 0,  \\
        \tfrac{\pi}{2} & \text{ otherwise.}
    \end{array}
    \right.
\end{equation*}
We are interested in two distinct types of waves:
\begin{itemize}
    \item \textit{transverse waves} (i.e. $\kappa_y = 0$ and $\kappa_x \neq 0$) are perpendicular to the direction of movement;
    \item \textit{divergent waves} (i.e. $\kappa_y \neq 0$) are oblique relative to the $\bm{e}_x-$axis.
\end{itemize}
As observed in the following, existence of such waves is characterized by the critical velocity \eqref{eq:Uc}.
Thus, we want to solve the equation 
\begin{equation}
    \label{eq:kcritique2D}
2\pi \kappa_x\Ux=\omega_{\kfourier},
\end{equation}
for $\kappa_x>0$, in view of $\mathcal{D}$ and also
because it is directly related to the situation where the denominator of the stationary solution  $\mudiff_{\kfourier,0}^\star$ vanishes. 
 For every value of $\Ux>0$, we will see that there is a curve of solutions to~\eqref{eq:kcritique2D}, so that  $\mudiff_{\kfourier,0}^\star$ never belongs to $L^2(\mathbb{R}^2)$ (for a generic $f$). Consequently, there is always a wake behind the ship. This is a major difference with the 1D case. In Carusotto and Rousseaux~\cite{carusotto2013}, equations similar to~\eqref{eq:kcritique2D} have been thoroughly studied to gain some insight into the wave pattern.

Let $U_c$ be defined by~\eqref{eq:Uc}. We look for $\kfourier$ in polar coordinates $(r,\theta)$ with $r=\normeuclide{\kfourier}>0$. Up to symmetry, we remark that
\begin{equation*}
    \theta = \pm  \left\lbrace  
    \begin{array}{cc}
        \arctan \left( \tfrac{\kappa_y}{\kappa_x} \right) & \text{ if } \kappa_x \neq 0,  \\
        0 & \text{ else.}
    \end{array}
    \right.
\end{equation*}
so that $\varphi + \theta = \pm \tfrac{\pi}{2}$. The geometrical configuration and the different waves occurring are illustrated in Figure~\ref{fig:transverse_divergent_waves}.
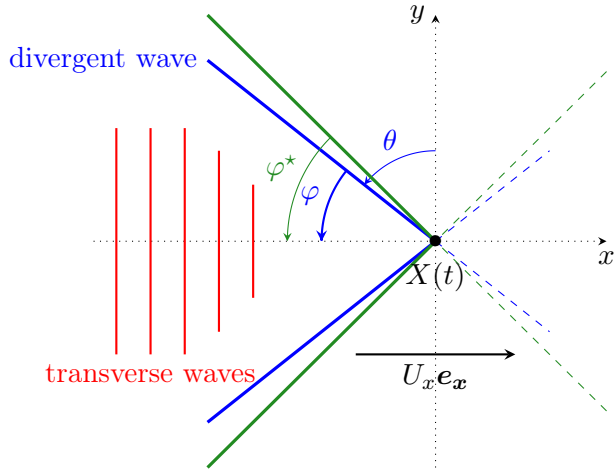
\begin{figure}[ht]
    \centering
    \begin{tikzpicture}[scale=1.5]
        \draw [dotted, , >=stealth,->](-3,0) -- (1.5,0);
        \draw [below](1.5,0) node{$x$};
        \draw [dotted, , >=stealth,->] (0,-2) -- (0,2);
        \draw [left](0,2) node{$y$};
        \draw [dashed, blue] (0,0) -- (1,.8);
        \draw [very thick, blue] (0,0) -- (-2,-1.6);
        \draw [blue, thick, domain=141.34:180,samples=100,>=stealth,->] plot ({cos(\x)}, {sin(\x)});
        \draw [blue] (-1.1,.4) node{$\varphi$};
        \draw [blue, domain=90:141.34,samples=100,>=stealth,->] plot ({.8*cos(\x)}, {.8*sin(\x)});
        \draw [blue] (-.4,.9) node{$\theta$};
        \draw [dashed, blue] (0,0) -- (1,-.8);
        \draw [very thick, blue] (0,0) -- (-2,1.6);
        \draw [left, blue] (-2,1.6) node{divergent wave};
        \draw [color=mygreen, very thick] (-2,-2) -- (0,0);
        \draw [color=mygreen, dashed] (1.5,-1.5) -- (0,0);
        \draw [color=mygreen, very thick] (-2,2) -- (0,0);
        \draw [color=mygreen, dashed] (1.5,1.5) -- (0,0);
        \draw [color=mygreen, domain=135:180,samples=100,>=stealth,->] plot ({1.3*cos(\x)}, {1.3*sin(\x)});
        \draw [color=mygreen] (-1.35,.65) node{$\varphi^{\star}$};
        \draw [thick, red] (-2.8,-1) -- (-2.8,1);
        \draw [thick, red] (-2.5,-1) -- (-2.5,1);
        \draw [thick, red] (-2.2,-1) -- (-2.2,1);
        \draw [thick, red] (-1.9,-.8) -- (-1.9,.8);
        \draw [thick, red] (-1.6,-.5) -- (-1.6,.5);
        \draw [below, red] (-2.5,-1) node{transverse waves};
        \draw (0,0) node{$\bullet$};
        \draw [below] (0,-.1) node{$X(t)$};
        \draw [thick, , >=stealth,->](-.7,-1) -- (.7,-1);
        \draw [below] (0,-1) node {$U_x \bm{e_x}$};
    \end{tikzpicture}
    \caption{Illustration of two dimensional wave propagation. The lines correspond to straight lines along which simple waves are constant. Blue ones form an angle $\varphi$  with the $x-$axis and correspond to divergent waves. Transverse waves (red lines) occur only in a subcritical configuration. In the supercritical case, divergent waves are contained within a cone with angle $\varphi^{\star}$ (green line).}
    \label{fig:transverse_divergent_waves}
\end{figure}

In the supercritical case $\Ux>U_c$, equation~\eqref{eq:kcritique2D} reads
\begin{equation*}
    2\pi r \Ux \cos\theta=\omega_{\kfourier}\iff \cos\theta=\frac{v_p(r)}{\Ux}\iff\theta=\pm\arccos \left(\frac{v_p(r)}{\Ux}\right),\quad r\in(0,+\infty),
\end{equation*}
where $v_p(\normeuclide{\kfourier})=\tfrac{\omega_{\kfourier}}{2\pi\normeuclide{\kfourier}}$ is the phase velocity. So $\sin \varphi = \sin ( \pm \pi/2 - \theta) = \cos \theta$ and we deduce (with symmetries):
\begin{equation*}
    \varphi = \pm \arcsin \left( \dfrac{v_p(r)}{U_x} \right).
\end{equation*} 
By Theorem~\ref{thm:oneD_wave}, $r \mapsto \tfrac{v_p(r)}{\Ux}$ is continuously and strictly decreasing from $U_c/\Ux<1$ to $0^+$ on the interval $(0,+\infty)$, so that $\theta$ exist and $r\mapsto\arccos(v_p(r)/\Ux)$ is strictly increasing from a minimum angle $\theta^\star=\arccos(U_c/\Ux)$ to $\pi/2$. So, simple waves are contained within a cone forming an angle 
\begin{equation}
\label{varphi_star}
    \varphi^{\star} = \pm \arcsin \left( \dfrac{U_c}{\Ux} \right).
\end{equation}
The expression $\theta=\theta(r)$ we found allows to draw $\mathcal{D}$ and the curve $r \mapsto \varphi(r)$, as illustrated in Figure~\ref{fig:supercritical2D}. 

\begin{figure}[ht]
\centerline{\includegraphics{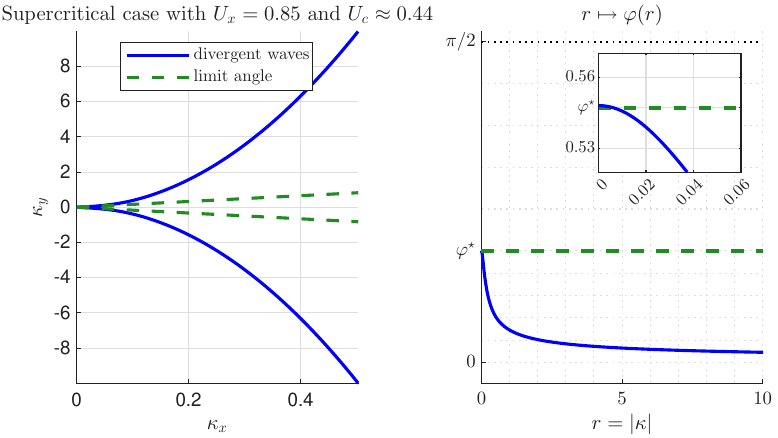}}
\caption{\label{fig:supercritical2D}Supercritical case  $\Ux>U_c$ using $\rho_1 = 999~\si{kg \cdot m^{-3}}$, $\rho_2 = 1022.3~\si{kg \cdot m^{-3}}$, $h_1 = 1~\si{m}$, $h_2 = 6~\si{m}$ and $g = 9.81~\si{m \cdot s^{-2}}$. Left: Domain $\mathcal{D}$, the line of singularities is the blue curve while the green dashed lines define the limit angle $\varphi^{\star}$. Right : curve $r \mapsto \varphi(r)$ is the blue line which is bounded by $0$ and $\varphi^{\star} \approx 0.5470$.}
\end{figure}
   
In the subcritical case $\Ux<U_c$, equation~\eqref{eq:kcritique2D} reads again $\cos\theta=v_p(r)/\Ux$. Since $v_p(r)/\Ux$ is continuously and strictly decreasing from $U_c/\Ux>1$ to $0^+$ on $(0,+\infty)$, there is a unique $r^\star>0$ such that $v_p(r^\star)=\Ux$ (cf. Theorem~\ref{thm:oneD_wave}). The value $r^{\star}$ corresponds to a transverse wave. For $r\in[r^\star,+\infty)$, the singularity curve is defined by
\begin{equation*}
    \theta(r)=\pm\arccos \left(\dfrac{v_p(r)}{\Ux}\right),\quad r\in[r^\star,+\infty).
\end{equation*}
In Figure~\ref{fig:subcritical2D}, we plot example of the domain $\mathcal{D}$ and $r \mapsto \varphi(r)$ in a subcritical configuration.

\begin{figure}[ht]
\centerline{\includegraphics{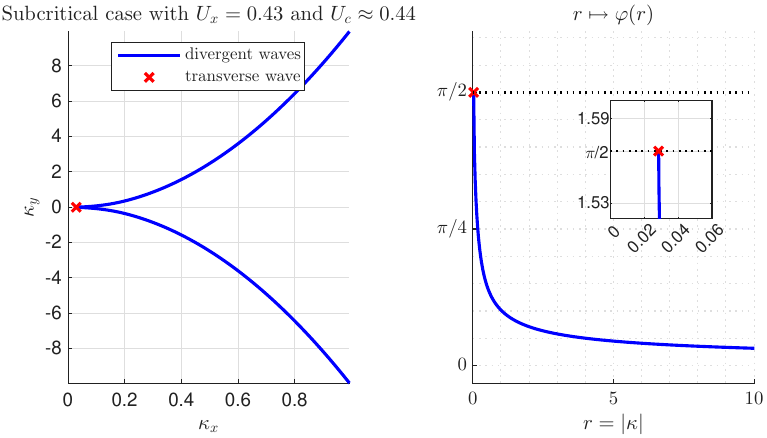}}
\caption{\label{fig:subcritical2D} Subcritical case  $\Ux<U_c$ using $\rho_1 = 999~\si{kg \cdot m^{-3}}$, $\rho_2 = 1022.3~\si{kg \cdot m^{-3}}$, $h_1 = 1~\si{m}$, $h_2 = 6~\si{m}$ and $g = 9.81~\si{m \cdot s^{-2}}$. Left: Domain $\mathcal{D}$ is plotted with the blue curve, the red cross corresponds to the transverse wave. Right : the curve $r \mapsto \varphi(r)$ corresponds to the blue line and the red cross to the transverse wave with $r^{\star} \approx 2.8589 \cdot 10^{-2}$ that depends on the chosen physical parameters.}
\end{figure}
   
\begin{remark}
    In the 2D case, it is also possible to find a limit for the stationary solution as $\varepsilon\to 0$, but the calculations are technical. In~\cite{Hudimac1961}, the Green function (that is, the solution for a point source) is computed for a closely related problem. However, Hudimac deals directly with the situation $\varepsilon=0$ and he handles the singular integrals with a Fourier single-integral limit theorem.
\end{remark}
\begin{remark}
    The approach considered allows to describe the wake pattern that will be observed in numerical simulations (see Section~\ref{sec:numan}). It has similarly been considered for single layer configurations. For instance in~\cite{rabaud2013ship, crawford1984elementary} the authors consider a geometric approach to describe the wake behind an object and a limit angle. 
\end{remark}

\section{Numerical analysis}
\label{sec:numan}
An explicit solution of~\eqref{eq:WaterWaves_viscous} is  given by the Proposition~\ref{prop:solUconst} if the velocity $\bmU$ is constant. However, the boat velocity is related to the physical context (engine speed, environmental resistance...). For this reason, in a general context, a numerical scheme is needed to solve~\eqref{eq:WaterWaves_viscous}. This is the topic of this section.

\subsection{Discrete Fourier transform}

Previously, the model~\eqref{eq:WaterWaves_viscous} was considered in an infinite domain $\Omega = \mathbb{R}^d$ using the Fourier transform. For numerical considerations, we replace $\Omega$ by a bounded periodic domain. In the one dimensional case, we now consider a domain with length $L_x > 0$ given by $\Omega=(-L_x/2;L_x/2)$. The two dimensional domain is $\Omega = (-L_x/2;L_x/2)\times (-L_y/2;L_y/2)$ and corresponds to a rectangle of dimensions $L_x \times L_y$ (with $L_x>0$ and $L_y>0$).

Due to periodicity, we can effectively handle boundary conditions and implement the numerical algorithms. The solution is evaluated on a periodic grid, and the Fast Fourier Transform (FFT) is used to analyze and manipulate the solution in the Fourier space, making the computation more efficient~\cite{jenkins2022fourier,ryan2019linear}. This approach transforms the problem from a infinite domain into a manageable finite one, where the periodicity implies that the solution repeats itself outside the boundaries of the computational domain.  Then numerical simulations should be done carefully to prevent the periodicity from deteriorating the solution by interacting with itself. For this reason, the simulations will be carried out with sufficiently large domains with respect to the final time.

\subsection{Time discretization}

In this subsection, we focus on the temporal discretization of the governing equations. Our goal is to compute $\mudiff_{\kfourier,\varepsilon}^{(k)} \approx \mudiff_{\kfourier,\varepsilon}(k \Delta t)$ with $\Delta t > 0$ the time step and $k \in \mathbb{N}$ the iterate.

A time integrator can be obtained by discretizing~\eqref{eq:WaterWaves_viscous}. However, discretizations using an explicit scheme are subject to time step constraints that make long-term simulations tricky. Conversely, some implicit methods do not have time step restrictions but suffer from poor dissipation and dispersion properties. For instance, the backward Euler method includes an extra dissipation which results in unexpected attenuation of the physical phenomena to be observed. Conversely, the Crank-Nicholson method does not preserve dispersion properties and waves are distorted from expected ones. To avoid both the time step constraints and poor dissipation and dispersion properties, we consider exponential integrators. A more complete comparison of time integrators to solve wave equation is available in~\cite{brachet2022comparison} in the context of non-dispersive PDEs. 

The exponential integrators we consider are obtained from~\eqref{eq:Duhamel} by considering quadrature rules on the integral. Formula~\eqref{eq:Duhamel} leads to the iterate relation
\begin{align}
    \mudiff_{\kfourier,\varepsilon}(t^{(k+1)}) & = \mudiff_{\kfourier,\varepsilon}(t^{(k)})e^{i \omega_{\kfourier,\varepsilon} \Delta t} - \displaystyle\int_{t^{(k)}}^{t^{(k+1)}} e^{i \omega_{\kfourier,\varepsilon}(t^{(k+1)} -\tau)} \hat{g}_{\kfourier}(\tau)d\tau \nonumber \\
        & = e^{i \omega_{\kfourier,\varepsilon} \Delta t} \left( \mudiff_{\kfourier,\varepsilon}(t^{(k)}) - \displaystyle\int_{0}^{\Delta t} e^{-i \omega_{\kfourier,\varepsilon}\tau} \hat{g}_{\kfourier}(\tau+t^{(k)})d\tau \right), \label{eq:iterationexacte}
\end{align}
where $t^{(k)}=k\Delta t$. 
Let  $\mathcal{Q}_{\kfourier,\Delta t}(t^{(k)})$ be a $p$-order accurate approximate value of the integral part such that
\begin{equation}
    \begin{vmatrix} \mathcal{Q}_{\kfourier,\Delta t}(t^{(k)}) - \displaystyle\int_{0}^{\Delta t} e^{-i \omega_{\kfourier,\varepsilon}\tau} \hat{g}_{\kfourier}(\tau+t^{(k)})d\tau \end{vmatrix} \leq C^{(p)}_{\kfourier} \Delta t^{p+1},
    \label{eq:quadraturep}
\end{equation}
where $C^{(p)}_{\kfourier}$ is a constant that is independent of $\Delta t$ on a finite time interval $[0,T]$.
Then, for all $k \in \mathbb{N}$, $\mudiff_{\kfourier,\varepsilon}^{(k)} \approx \mudiff_{\kfourier,\varepsilon}(k \Delta t)$ is computed with
\begin{equation}
    \label{eq:iterationquad}
    \mudiff_{\kfourier,\varepsilon}^{(k+1)} = e^{i \omega_{\kfourier,\varepsilon} \Delta t} \left( \mudiff_{\kfourier,\varepsilon}^{(k)} - \mathcal{Q}_{\kfourier,\Delta t}(t^{(k)}) \right).
\end{equation}
An estimate such as~\eqref{eq:quadraturep} can be obtained by a numerical integration of order $p$ if $t\mapsto \hat{g}_{\kfourier}(t)$ is of class $C^{p+1}$ on $[0,T]$. This happens if the velocity $\bmU$ is of class $C^{p+1}$ on $[0,T]$, by~\eqref{eq:defgk}. In that case, the constant $C^{(p)}_{\kfourier}$ depends on the derivatives of $\hat{g}_{\kfourier}$ up to order $p+1$ on $[0,T]$ and on $\varepsilon T$.

We deduce the following accuracy result.
\begin{proposition}
    \label{prop:accuracy}
    Let $\varepsilon\ge 0$, $T>0$ and assume that $\mudiff_{\kfourier,\varepsilon}^{(0)} = \mudiff_{\kfourier,\varepsilon}(0)$. Then for all $k \geq 0$, $\Delta t>0$ and $\kfourier$ such that $k\Delta t\le T$, we have
    \begin{equation*}
        \begin{vmatrix} \mudiff_{\kfourier,\varepsilon}^{(k)} - \mudiff_{\kfourier,\varepsilon}(k\Delta t) \end{vmatrix} \leq T C_{\kfourier}^{(p)} \Delta t^{p}.
    \end{equation*}
\end{proposition}
\begin{proof}
Let $\ell \geq 0$ such that $(\ell+1)\Delta t\le T$. Considering~\eqref{eq:iterationexacte} and~\eqref{eq:iterationquad}, by difference, we get
\begin{align*}
    \begin{vmatrix} \mudiff_{\kfourier,\varepsilon}^{(\ell+1)} - \mudiff_{\kfourier,\varepsilon}((\ell+1) \Delta t) \end{vmatrix} & \leq e^{-\varepsilon \Delta t} \begin{vmatrix} \mudiff_{\kfourier,\varepsilon}^{(\ell)} - \mudiff_{\kfourier,\varepsilon}(\ell \Delta t) \end{vmatrix} + C_{\kfourier}^{(p)} \Delta t^{p+1} \\
        & \leq \begin{vmatrix} \mudiff_{\kfourier,\varepsilon}^{(\ell)} - \mudiff_{\kfourier,\varepsilon}(\ell \Delta t) \end{vmatrix} + C^{(p)}_{\kfourier} \Delta t^{p+1}.
\end{align*}
So we obtain $\begin{vmatrix} \mudiff_{\kfourier,\varepsilon}^{(\ell+1)} - \mudiff_{\kfourier,\varepsilon}((\ell+1) \Delta t) \end{vmatrix} - \begin{vmatrix} \mudiff_{\kfourier,\varepsilon}^{(\ell)} - \mudiff_{\kfourier,\varepsilon}(\ell \Delta t) \end{vmatrix} \leq C^{(p)}_{\kfourier} \Delta t^{p+1}$ whatever $\ell \geq 0$.
Reminding that $\mudiff_{\kfourier,\varepsilon}^{(0)} = \mudiff_{\kfourier,\varepsilon}(0)$ and thanks to a telescoping sum, for $k\Delta t\le T$ we get
\begin{align*}
      \begin{vmatrix}
        \mudiff_{\kfourier,\varepsilon}^{(k)} - \mudiff_{\kfourier,\varepsilon}(k \Delta t) 
      \end{vmatrix} & = \displaystyle\sum_{\ell = 1}^{k} \begin{vmatrix} \mudiff_{\kfourier,\varepsilon}^{(\ell)} - \mudiff_{\kfourier,\varepsilon}((\ell) \Delta t) \end{vmatrix} - \begin{vmatrix} \mudiff_{\kfourier,\varepsilon}^{(\ell-1)} - \mudiff_{\kfourier,\varepsilon}((\ell-1) \Delta t) \end{vmatrix} \\
        & \leq \displaystyle\sum_{\ell = 1}^{k} C^{(p)}_{\kfourier} \Delta t^{p+1} \\
        & \le k\Delta t C_{\kfourier}^{(p)} \Delta t^{p}.
\end{align*}
 The claim follows.
\end{proof}

\begin{remark}
Proposition~\ref{prop:accuracy} gives a frequency-by-frequency estimate of the error. Numerically, we work with a finite number of frequencies $\kfourier$ related to the number of grid points. Then, Parseval's identity gives the following estimate:
\begin{equation*}
        \begin{Vmatrix} \mu_{\varepsilon}^{(k)} - \mu_{\varepsilon}(k\Delta t) \end{Vmatrix}_{\ell^2} \leq t^{(k)} C^{\sfrac{1}{2}} \Delta t^{p}
    \end{equation*}
    where $C = \sum_{\kfourier} C_{\kfourier}^{(p)}$ is a finite sum. However, in this context, an error can occur due to high frequencies which are truncated. These components are damped by the diffusion parameter $\varepsilon>0$.
\end{remark}

Several quadrature rules can be used to compute $\mathcal{Q}_{\kfourier,\Delta t}(t^{(k)})$. For illustration, we mention
\begin{itemize}
    \item the trapezoidal rule :
    \begin{equation*}
        \mathcal{Q}_{\kfourier,\Delta t}(t^{(k)}) = \dfrac{\Delta t}{2} \left( \hat{g}_{\kfourier}(t^{(k)}) +  e^{- i \omega_{\kfourier,\varepsilon} \Delta t}\hat{g}_{\kfourier}(t^{(k+1)})  \right)
    \end{equation*}
    which is second order accurate;
    \item the fourth order accurate Simpson rule :
    \begin{equation*}
        \mathcal{Q}_{\kfourier,\Delta t}(t^{(k)}) = \dfrac{\Delta t}{6} \left( \hat{g}_{\kfourier}(t^{(k)}) +  4e^{- i \omega_{\kfourier,\varepsilon} \sfrac{\Delta t}{2}}\hat{g}_{\kfourier}(t^{(k+\sfrac{1}{2})}) +  e^{- i \omega_{\kfourier,\varepsilon} \Delta t}\hat{g}_{\kfourier}(t^{(k+1)})  \right).
    \end{equation*}
\end{itemize}
The use of one of the previous formulas requires to compute $\hat{g}_{\kfourier}(t)$ for $t > t^{(k)}$. More precisely, we need to know the speed of the boat $\bm{U}(t)$ for some $t > t^{(k)}$. In this paper, the velocity is assumed known but in a more physical context, it could be given by an other equation. For this reason, we restrict our simulations to the rectangle rule $\mathcal{Q}_{\kfourier,\Delta t}(t^{(k)}) = \Delta t \hat{g}_{\kfourier}(t^{(k)})$. So, the sequence $(\mudiff_{\kfourier,\varepsilon}^{(k)})$ is computed with
\begin{equation}
    \mudiff_{\kfourier,\varepsilon}^{(k+1)} = e^{i \omega_{\kfourier,\varepsilon}} \left( \mudiff_{\kfourier,\varepsilon}^{(k)} - \Delta t \hat{g}_{\kfourier}(t^{(k)}) \right).
    \label{eq:ERK1_WW}
\end{equation}
This time scheme is first order accurate and the velocity should be known only at the current time $t^{(k)}$ to compute the next iterate.

The fact that the error is controlled by Proposition~\ref{prop:accuracy} does not mean that the method is stable in the sense of the evolution of the errors introduced at a given time (e.g. $t=0$). To do this, we consider two initial conditions $\mudiff_{\kfourier,\varepsilon}^{1}$ and $\mudiff_{\kfourier,\varepsilon}^{2}$ and analyze the way $\tilde{\mu}_{\kfourier,\varepsilon} = \mudiff_{\kfourier,\varepsilon}^{1} - \mudiff_{\kfourier,\varepsilon}^{2}$ evolves with time. In continuous time, we have
\begin{equation*}
    \tilde{\mu}_{\kfourier,\varepsilon}(t) = e^{i \omega_{\kfourier,\varepsilon} t} \tilde{\mu}_{\kfourier,\varepsilon}(0).
\end{equation*}
We deduce the continuous $L^2$ dissipation-stability result : $\begin{Vmatrix} \tilde{\mu}_{\varepsilon}(t) \end{Vmatrix}_{L^2} = e^{- \varepsilon t} \begin{Vmatrix} \tilde{\mu}_{\varepsilon}(0) \end{Vmatrix}_{L^2}$. This means the difference of two initial functions (error initially introduced) goes to zero when $t$ increases. About the dispersion-stability, we have $\arg \left( \tilde{\mu}_{\kfourier,\varepsilon}(t) \right) = \omega_{\kfourier} t$ which prescribes the way in which each frequency component evolves.
This properties are exactly recovered by using exponential integrator (whatever the quadrature rule $\mathcal{Q}_{\kfourier, \Delta t}$ used). Indeed, we have the discrete property $\tilde{\mu}_{\kfourier,\varepsilon}^{(k+1)} = e^{i \omega_{\kfourier,\varepsilon}} \tilde{\mu}_{\kfourier,\varepsilon}^{(k)}$
for all $k \in \mathbb{N}$. So we deduce:
\begin{align*}
    \begin{Vmatrix} \tilde{\mu}_{\varepsilon}^{(k)} \end{Vmatrix}_{L^2} & = e^{- \varepsilon k \Delta t} \begin{Vmatrix} \tilde{\mu}_{\varepsilon}(0) \end{Vmatrix}_{L^2}, \\
    \arg \left( \tilde{\mu}_{\kfourier,\varepsilon}^{(k)} \right) & = \omega_{\kfourier} k \Delta t.
\end{align*}
The exponential integrator preserves dissipation and dispersion errors initially introduced in $\tilde{\mu}_{\kfourier,\varepsilon}(0)$ in the same way as the continuous equation. This is not the case when using the forward Euler or the Crank-Nicholson time schemes for which extra-dissipation and distortions are introduced. 

\section{Numerical experiments}

In this section, we present numerical simulations in one and two space dimensions  to examine the consistency between theoretical and numerical results.

To conduct experiments, we consider a fluid domain made up of two layers. The top layer has a thickness of $h_1 = 1~\si{m}$ and a density $\rho_1 = 999~\si{kg \cdot m^{-3}}$ which could corresponds to fresh water at ambient temperature. Otherwise, the bottom layer's thickness is $h_2 = 6~\si{m}$ and it is composed of salty water of density $\rho_2 = 1022.3~\si{kg \cdot m^{-3}}$. As $g = 9.81~\si{m \cdot s^{-2}}$, the critical velocity~\eqref{eq:Uc} is $U_c \approx 0.4421~\si{m \cdot s^{-1}}$.
Simulations are conducted in a domain with of dimension $L_x$ and $L_y$. Note that $L_y$ is only used for two dimensional simulations.
The size of the domain will be specified in numerical experiments.

The boat's function is 
\begin{equation*}
    f(\bm{x}) = \left\lbrace
    \begin{array}{rl}
        - T \exp \left( -18 \left( \tfrac{x}{L_1} \right)^2 \right) & \text{for 1D experiments,} \\
        - T \exp \left( -18 \left( \left( \tfrac{x}{L_1} \right)^2+ \left( \tfrac{y}{L_2} \right)^2 \right) \right) & \text{for 2D experiments}.
    \end{array}
    \right.
\end{equation*}
where $L_1$ and $L_2$ are the length and the beam of the boat. The value $T$ is the boat's draft. In experiments, we consider $L_1 = L_2 = 10~\si{m}$ and $T = 0.02~\si{m}$. The boat sails in the direction of the $x-$axis at a speed given by $\bmU(t) = \Ux(t) \bm{e_x}$ where $\Ux(t)$ is specified in experiments.

For simulations where a non-zero value $\varepsilon$ is required, we compute $\varepsilon$ as small as possible using the strategy detailed in Appendix~\ref{app:exponentialintegrator}.

\begin{remark}
When the velocity $\bmU$ is constant, an analytical solution is provided by Proposition~\ref{prop:solUconst}. In particular, when the initial interface $\mu_{\varepsilon}(0)$ is equal to the steady state $\mu_{\varepsilon,0}^\star$, the deformation remains constant  over time (in the boat's frame of reference) and an analytical solution is also available, cf.~\eqref{eq:initial_steady}. However, there is no analytical solution in general if the velocity $\bmU$ is time dependent, as it is the case in a more realistic context. For the sake of clarity, the  solution computed at time $t^{N} = N \Delta t$ using a numerical scheme is denoted by $\eta_\varepsilon^N$ while $\eta_\varepsilon(t^{N})$ is the analytical solution (if available). Note that $\eta_\varepsilon(t^{N})$ is computed with a FFT for the space discretization.
\end{remark}

\subsection{One dimensional simulations}

We start with one dimensional simulations to analyze the numerical  properties of the scheme and the  behaviour of the internal waves.

\begin{example}
\label{ex:converence}
The first numerical experiment  is performed to validate the scheme by analyzing the convergence order of the numerical scheme, assuming that the boat moves at a constant speed. The velocity is given by $\Ux(t)  = 0.43~\si{m \cdot s^{-1}}$, and the initial condition is taken as $\mu_{\varepsilon}(0)=\mu_{\varepsilon,0}^\star$ (see~\eqref{eq:initial_steady}), so that $\mu_{\varepsilon}$ is the steady state solution $\mu_{\varepsilon}^\star$ (in the boat's referential). 
The length of the domain is $L_x = 12 \ 000~\si{m}$. The simulation runs over a final time of $t^N = 4000~\si{s}$, using a refined spatial discretization ($N_x = 15000$ points, corresponding to $\Delta x = 0.8~\si{m}$). Different time steps $\Delta t$ are used to examine temporal convergence. For each configuration, the relative error in the $\ell^2$ norm between the numerical solution and the corresponding theoretical reference from Proposition~\ref{prop:solUconst} is computed.

Figure~\ref{fig:convergence_avec_epsilon} illustrates the evolution of the relative $\ell^2$ error $\tfrac{\| \eta_{\varepsilon}(t^N) - \eta_{\varepsilon}^N \|_{\ell^2}}{\| \eta_{\varepsilon}(t^N) \|_{\ell^2}}$ between the numerical and theoretical solutions as the time step  $\Delta t$ is progressively decreased. Several values $\varepsilon$ are used. 
\begin{figure}[H]
    \centering
    \includegraphics[scale=1]{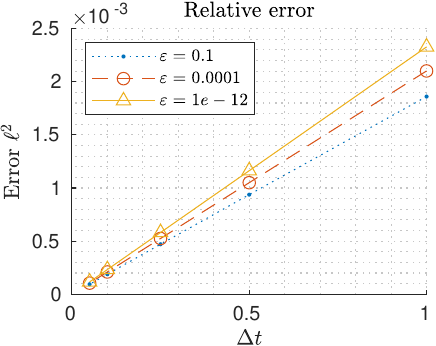}
  \caption{Example~\ref{ex:converence}. Relative $\ell^2$ error $\tfrac{\| \eta_{\varepsilon}(t^N) - \eta_{\varepsilon}^N \|_{\ell^2}}{\| \eta_{\varepsilon}(t^N) \|_{\ell^2}}$ between the numerical and analytical solutions for various values of $\varepsilon \in \{ 10^{-12}, 10^{-4}, 10^{-1} \}$ in the subcritical regime.}
    \label{fig:convergence_avec_epsilon}
\end{figure}
The results confirm the first order of accuracy expected from Proposition~\ref{prop:accuracy}. Using least square approximations, with $\varepsilon = 0.1$, the  estimated order is  approximately to $0.99522$, while for $\varepsilon = 0.0001$ it reaches $1.00029$ and it becomes  $1.00032$ with $\varepsilon = 10^{-12}$. 
\end{example}
\begin{example}
\label{ex:rayleigh}
In this simulation, we analyze the effect of the parameter $\varepsilon$. The boat moves at a constant velocity $U_x(t) = 0.43~\si{m \cdot s^{-1}}$ and the initial condition is given as previously by $\mu_{\varepsilon,0}^\star$. The domain has a total length $L_x = 6000~\si{m}$ and it is discretized using $N_x = 15  000$ points, yielding a spatial step $\Delta x = 0.4~\si{m}$. For each tested value of $\varepsilon$, the interface deformation $\eta_\varepsilon(t^N)$ is analyzed at the final time $t^N = 2000~\si{s}$.
\end{example}
The interface of the analytical solution $\eta_{\varepsilon}(t^N)$ is plotted on Figure~\ref{fig:tests_epsilon_mu0_0_h} with three different values of the regularization parameter $\varepsilon$. 
\begin{figure}[H]
    \centering
    \includegraphics[scale=1]{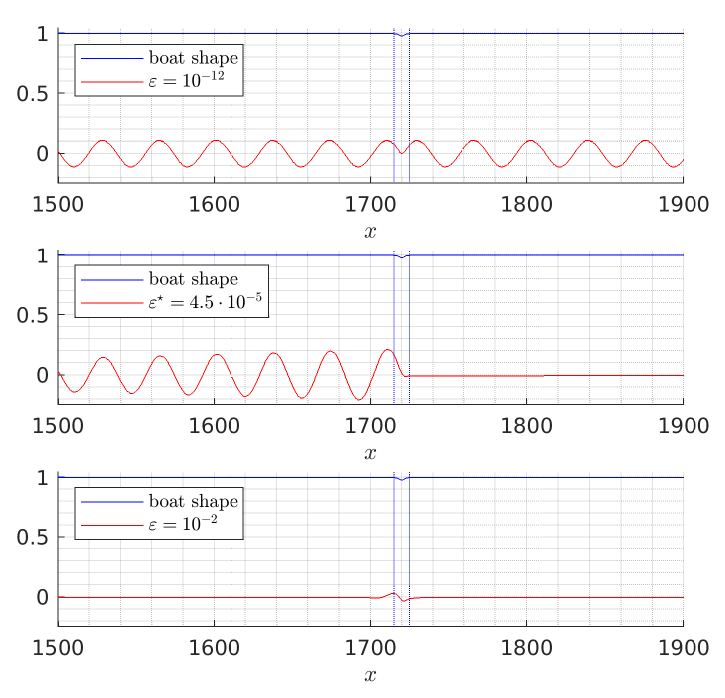}
    \caption{Example~\ref{ex:rayleigh}. Comparison of interface profiles for different regularization values of $\varepsilon$ at final time $t^N = 2000~\si{s}$ in a subcritical regime. The steady state solution is computed with the analytical expression.}
    \label{fig:tests_epsilon_mu0_0_h}
\end{figure}
For $\varepsilon = 10^{-12}$, we observe non-physical oscillations in front of the boat resulting from the periodicity of the domain. It can be reduced using a larger value $\varepsilon$. In contrast, the solution for $\varepsilon = 10^{-2}$ leads to a smooth and stable interface as expected but the regularization is too strong and the wake is immediately dissipated behind the boat. A compromise is obtained with Algorithm~\ref{alg:varepsilon} (cf. Appendix~\ref{app:exponentialintegrator}) and the value $\varepsilon^* = 4.5 \times 10^{-5}$ yields a deformation that is both smooth and stable while preserving the expected wake. This comparison highlights the importance of selecting an appropriate $\varepsilon$ to preserve the physical structure of the interface while avoiding periodicity artifacts.

\begin{example}
\label{ex:waves1d} 
The third experiment is an illustration of the supercritical and subcritical regimes highlighted in Theorem~\ref{thm:oneD_wave}. The interface is initially considered to be at rest with $\eta_{\varepsilon}(0) = 0$ and then disturbed by the movement of the boat. The damping parameter is $\varepsilon = 10^{-12}$. The velocities\footnote{We consider this kind of velocity to incorporate a progressive acceleration of the boat. With the parameters chosen, the speed reaches $99\%$ of $U_{\infty}$ at $t\approx 460.52~\si{s}$ and is close to be a constant thereafter.} are given by $U_x(t) = U_{\infty}(1-e^{-0.01 t})$ with $U_{\infty} = 0.25~\si{m \cdot s^{-1}}$ (subcritical) or $U_{\infty} = 0.65~\si{m \cdot s^{-1}}$ (supercritical). 
\end{example}
In Figure~\ref{fig:subcritical1D}, we display the numerical solution $\eta_\varepsilon^N$  at $t^N = 4000~\si{s}$ in the subcritical case. We observe a simple wave behind the boat that corresponds to the critical frequency $\kfourier_c$ introduced in Theorem~\ref{thm:oneD_wave}. It travels at the same speed as the boat without deforming. In addition, there is an over-elevation of the interface under the ship,  sometimes called Nansen's wake~\cite{fourdrinoy2020dual}. We also observe two local waves which are moving faster than the boat: one at the front of the boat and the other at the back in the opposite direction. These seem to be moving at the same velocity and are visible at $x \approx \pm 1750~\si{m}$.
\begin{figure}[H]
    \centering
    \includegraphics{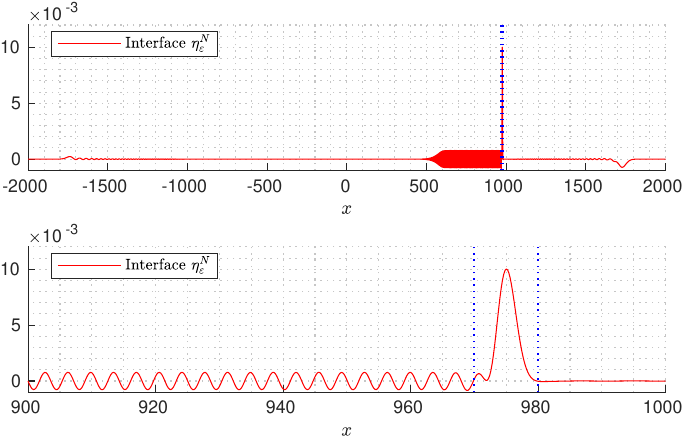}
   \caption{Example~\ref{ex:waves1d}. Interface deformation $\eta_{\varepsilon}^N$ at final time in the subcritical regime. The spatial domain has total length $L_x = 4000~\si{m}$ and is discretized using $N_x = 20  000$ points, corresponding to a spatial step of $\Delta x = 0.2~\si{m}$. The time step is $\Delta t = 0.5~\si{s}$.}
    \label{fig:subcritical1D}
\end{figure}

A similar simulation is conducted in the supercritical case and the corresponding interface $\eta_{\varepsilon}^N$ at time $t^N = 4000~\si{s}$ is shown on Figure~\ref{fig:supercritical1D}. Since for all $t$ large enough we have $U_x(t) > U_c$, Theorem~\ref{thm:oneD_wave} predicts that all waves are slower than the boat as it is observed in Figure~\ref{fig:supercritical1D}. There is also a deformation of the interface right below the ship which can be interpreted as Nansen's wake in the supercritical case.
\begin{figure}[H]
    \centering
    \includegraphics{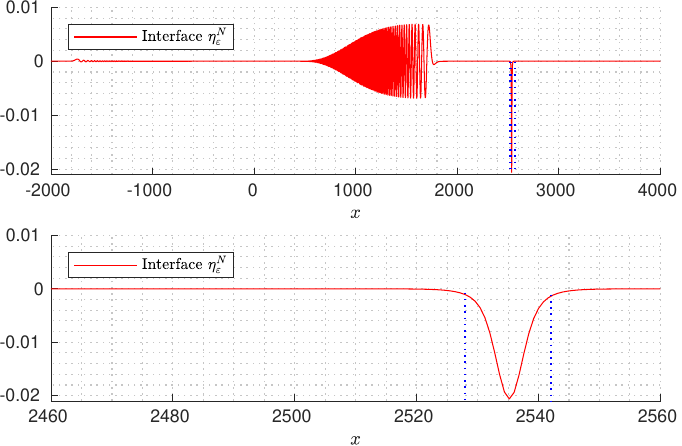}
  \caption{ Example~\ref{ex:waves1d}. Interface deformation $\eta_{\varepsilon}^N$ at final time in the supercritical regime. The simulation is conducted on a spatial domain of total length $L_x = 12  000~\si{m}$ with $N_x = 15  000$ points, yielding a spatial resolution of $\Delta x = 0.8~\si{m}$. The time step is set to $\Delta t = 0.5~\si{s}$. }
    \label{fig:supercritical1D}
\end{figure}

\begin{example}
\label{ex:waves1d-2} 
This experiment is similar to the one in Example~\ref{ex:waves1d}. But this time, we  analyze the supercritical and subcritical regimes using a space-time Fourier transform. The initial state is $\eta_{\varepsilon}(0) = 0$ and we consider $U_x(t) = U_{\infty}(1-e^{-0.01 t})$ with $U_{\infty} = 0.25~\si{m \cdot s^{-1}}$ (subcritical) or $U_{\infty} = 0.65~\si{m \cdot s^{-1}}$ (supercritical). The damping parameter is $\varepsilon = 10^{-12}$.
\end{example}

Figure~\ref{fig:double_fourier_vit_variabl12} is the double $(x,t)$-Fourier transform of $(x,t) \mapsto \eta_{\varepsilon}(x,t)$, where $\eta_{\varepsilon}$ is computed numerically at time $t^N = 4000~\si{s}$. The domain has a total length $L_x = 12 000~\si{m}$ and is discretised using $N_x = 15 000$ points, yielding a spatial resolution of $\Delta x = 0.8~\si{m}$. The time step is set to $\Delta t = 0.5~\si{s}$. The results highlight the dispersion relation $\omega_{\kfourier,\varepsilon}$ (dashed blue line) and the maximum boat velocity $2 \pi \kfourier U_{\infty}$ (dashed red line).
\begin{figure}[H]
    \centering
        \includegraphics[scale=.72]{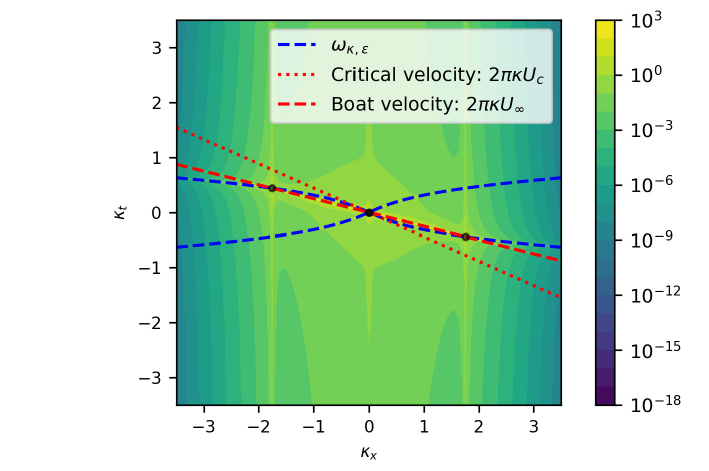}
        \includegraphics[scale=.72]{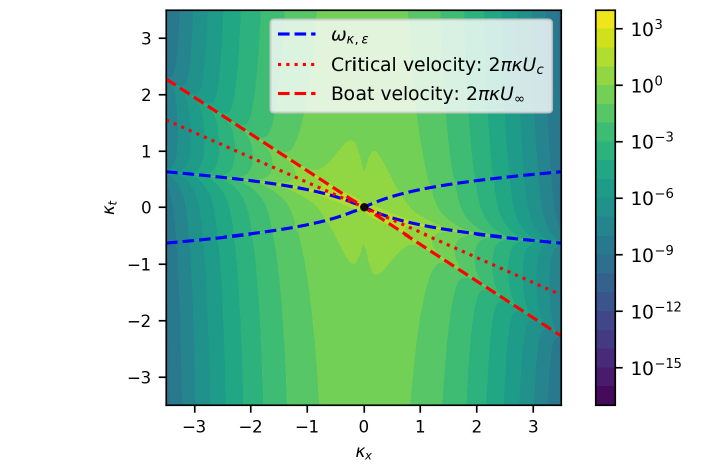} 
    \caption{Example~\ref{ex:waves1d-2}. Space-time Fourier transform of the internal wave.  On the left, the subcritical regime; on the right, the supercritical regime }
   \label{fig:double_fourier_vit_variabl12}
\end{figure}
The left-hand side of Figure~\ref{fig:double_fourier_vit_variabl12} corresponds to the subcritical regime. We observe that the line corresponding to the maximum boat velocity and the curve of the dispersion relation intersect at the abscissas $-\kfourier_c^{\star}$, $0$, and $\kfourier_c^{\star}$, as predicted by Theorem~\ref{thm:oneD_wave}. It means that some waves move faster, slower, or at the same velocity as $U_{\infty}$. 
In contrast, the right-hand side of Figure~\ref{fig:double_fourier_vit_variabl12} corresponds to the supercritical case. In this case, there is only one point of intersection at $0$. 
It is interesting to note that similar curves are obtained in \cite[Figure 2]{fourdrinoy2020dual} in an experimental context.
\begin{figure}[H]
    \centering
    \includegraphics[scale=0.4]{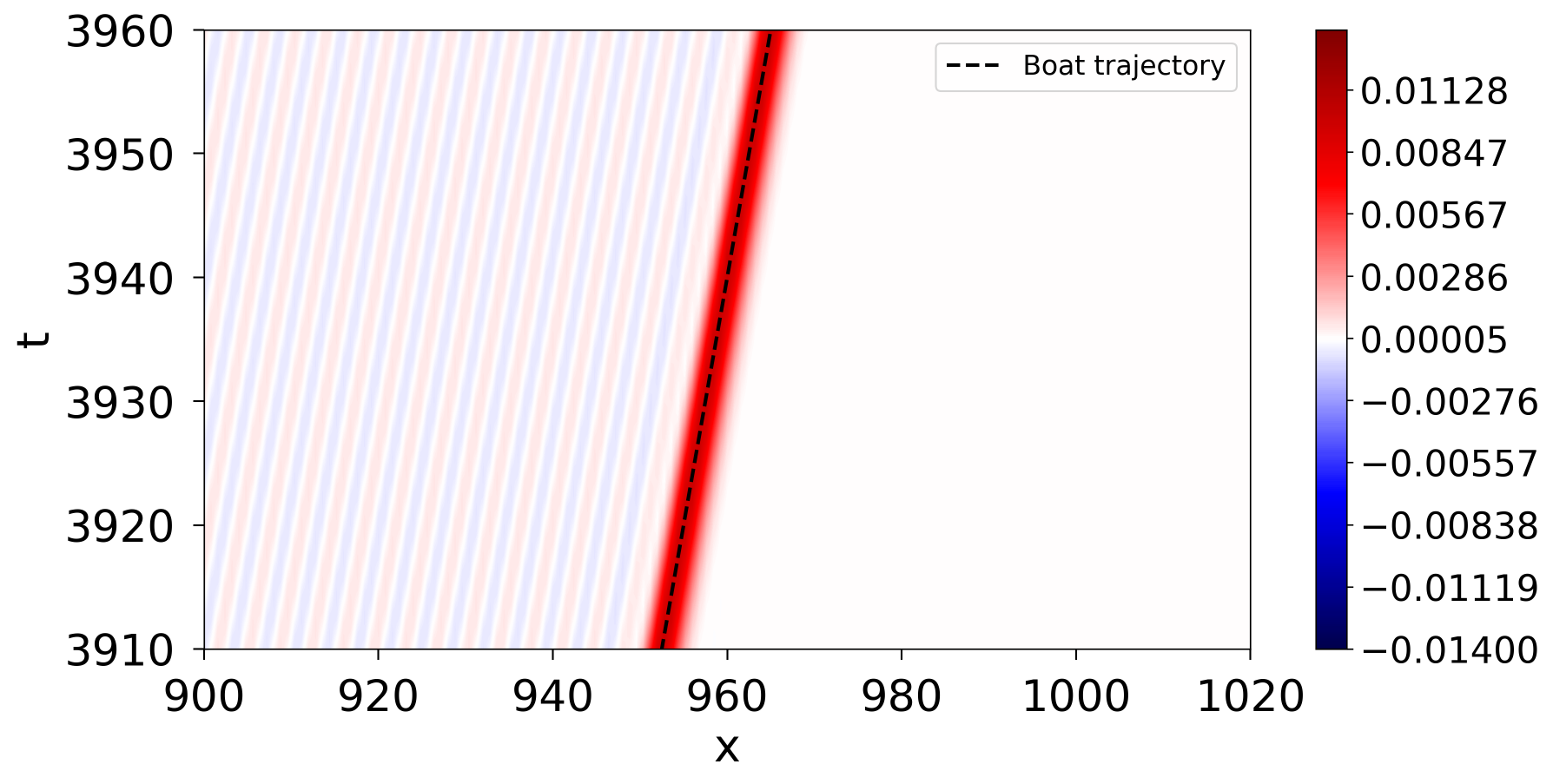}
    \caption{Example~\ref{ex:waves1d-2}. Interface deformation $\eta_{\varepsilon}^N$ in the subcritical regime, shown in the space-time domain.}
    \label{fig:subcritical1D_time}
\end{figure}
\FloatBarrier
Figure~\ref{fig:subcritical1D_time} shows the interface deformation over time in the subcritical case with variable velocity. The boat trajectory (black dashed line) is indicated, and oblique bands can be observed that represent the propagation of internal waves, mostly confined to the wake. It corresponds to the characteristic lines and confirms that a wave is moving with the same velocity than the boat.

\subsection{Two dimensional simulations}

This section deals with two dimensional simulation.  We consider the following experiments.

\begin{example}
\label{ex:2Dwaves}
This numerical experiment analyzes the evolution of the interface deformation in a two-dimensional configuration. The simulations are performed in both subcritical and supercritical regimes, with the boat speed set to   $U_x(t) = U_{\infty}(1-e^{-0.01 t})$, where $U_{\infty} = 0.43~\si{m \cdot s^{-1}}$ (subcritical) and $U_{\infty} = 0.85~\si{m \cdot s^{-1}}$ (supercritical). The damping parameter is fixed at  $\varepsilon = 10^{-12}$ in  both cases. These simulations are initialized with $\eta_{\varepsilon}(0)=0$ and performed using a time-dependent boat velocity.
\end{example}
We display the numerical solutions $\eta_{\varepsilon}^N$ at time $t^N = 2500~\si{s}$ in the subcritical and supercritical regimes. These simulations are performed on a spatial domain of length $L_x = 8000~\si{m}$ with $N_x = 12288$, 
giving a spatial resolution of $\Delta x = 0.65~\si{m}$. 
The domain in the $y$-direction has length $L_y = 2400~\si{m}$ with $N_y = 3072$, 
leading to a resolution of $\Delta y = 0.78~\si{m}$. 
The time step is $\Delta t = 0.5~\si{s}$.
\begin{figure}[ht]
    \centering
    \begin{minipage}[t]{0.48\textwidth}
        \centering
        \includegraphics[height=4.5cm]{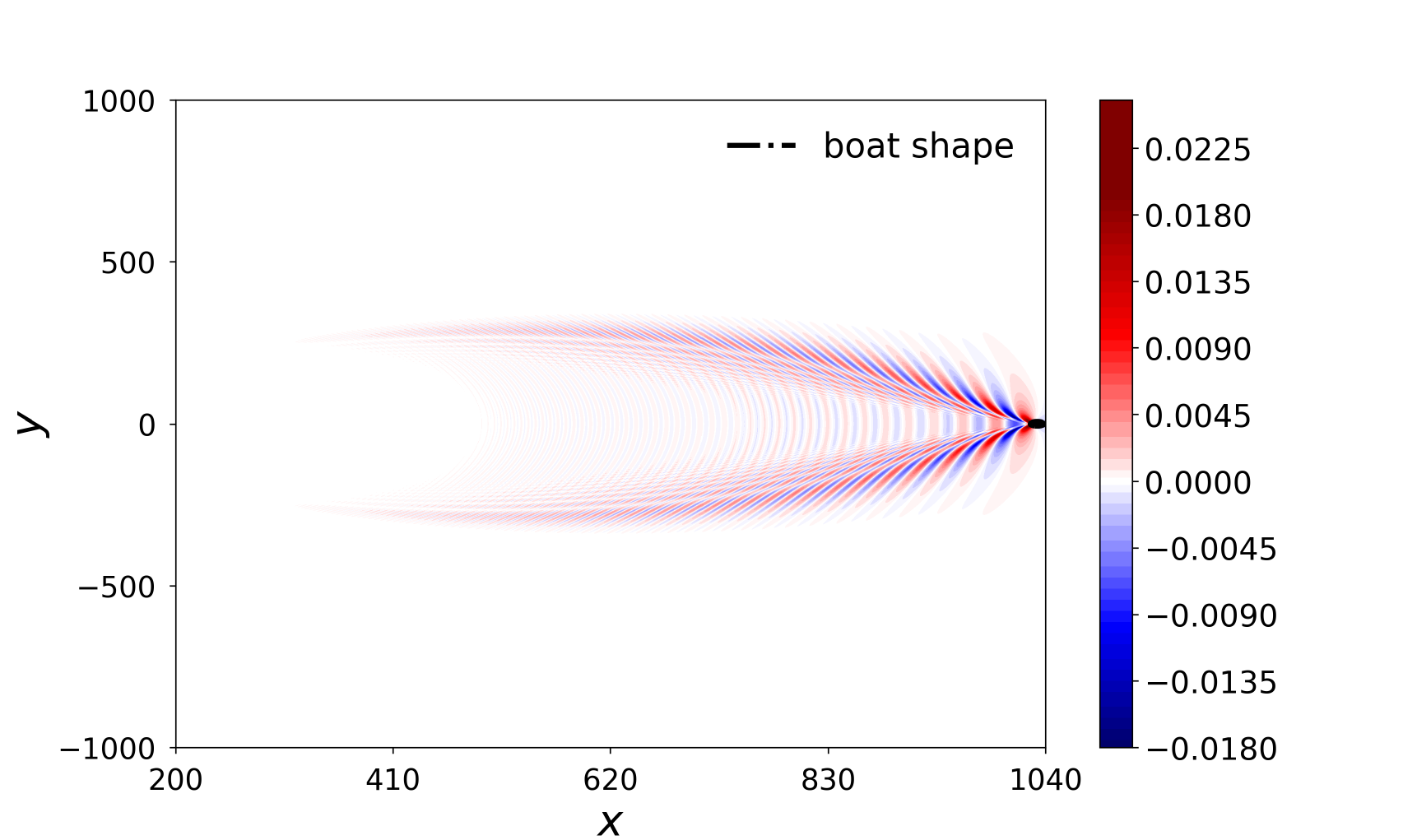}
    \end{minipage}
    \hfill
    \begin{minipage}[t]{0.48\textwidth}
        \centering
        \includegraphics[height=4.5cm]{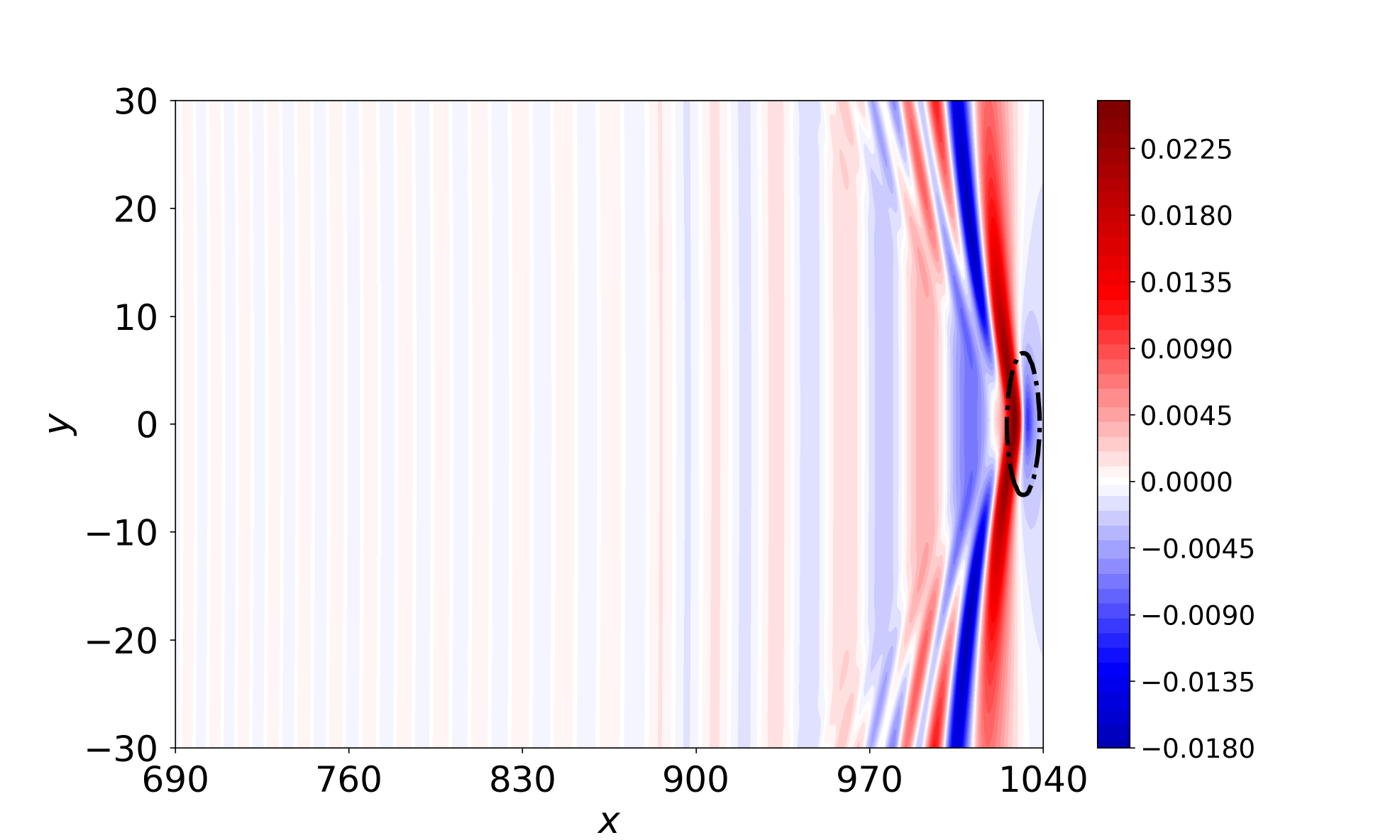}
    \end{minipage}
  \caption{Example~\ref{ex:2Dwaves}. Interface deformation  $\eta_{\varepsilon}^N$ at the final time  $t^N = 2500~\si{s}$ in the subcritical regime: the full two-dimensional configuration is shown on the left, while a zoomed view around the boat is displayed on the right.}
   \label{fig:cas_sous_crit_2dmu0est0}
\end{figure}
\FloatBarrier
\vspace{0.5em}
Figure~\ref{fig:cas_sous_crit_2dmu0est0} illustrates the interface deformation at the final time in the subcritical regime. The two-dimensional view (left) shows a well-structured wake developing behind the boat, whereas the interface remains regular in front of it. 
The zoomed-in view (right) highlights this behavior and shows the high amplitudes of the deformation concentrated in the region close to the vessel. We observe transverse waves (vertical lines) and divergent waves (oblique lines) as predicted in section~\ref{subsec:2D_theory}.
\begin{figure}[ht]
    \centering
    \begin{minipage}[t]{0.48\textwidth}
        \centering
        \includegraphics[height=4.5cm]{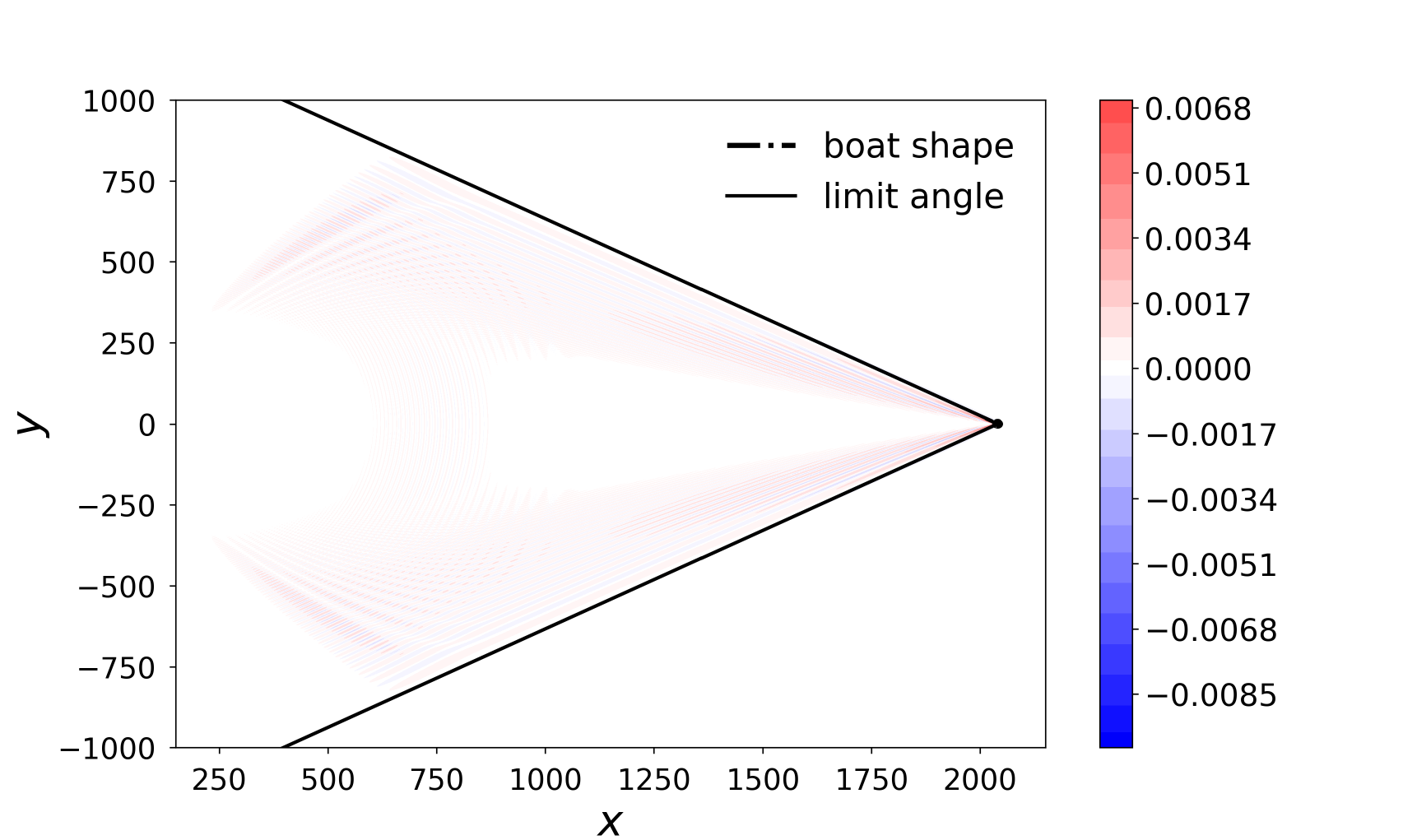}
    \end{minipage}
    \hfill
    \begin{minipage}[t]{0.48\textwidth}
        \centering
        \includegraphics[height=4.5cm]{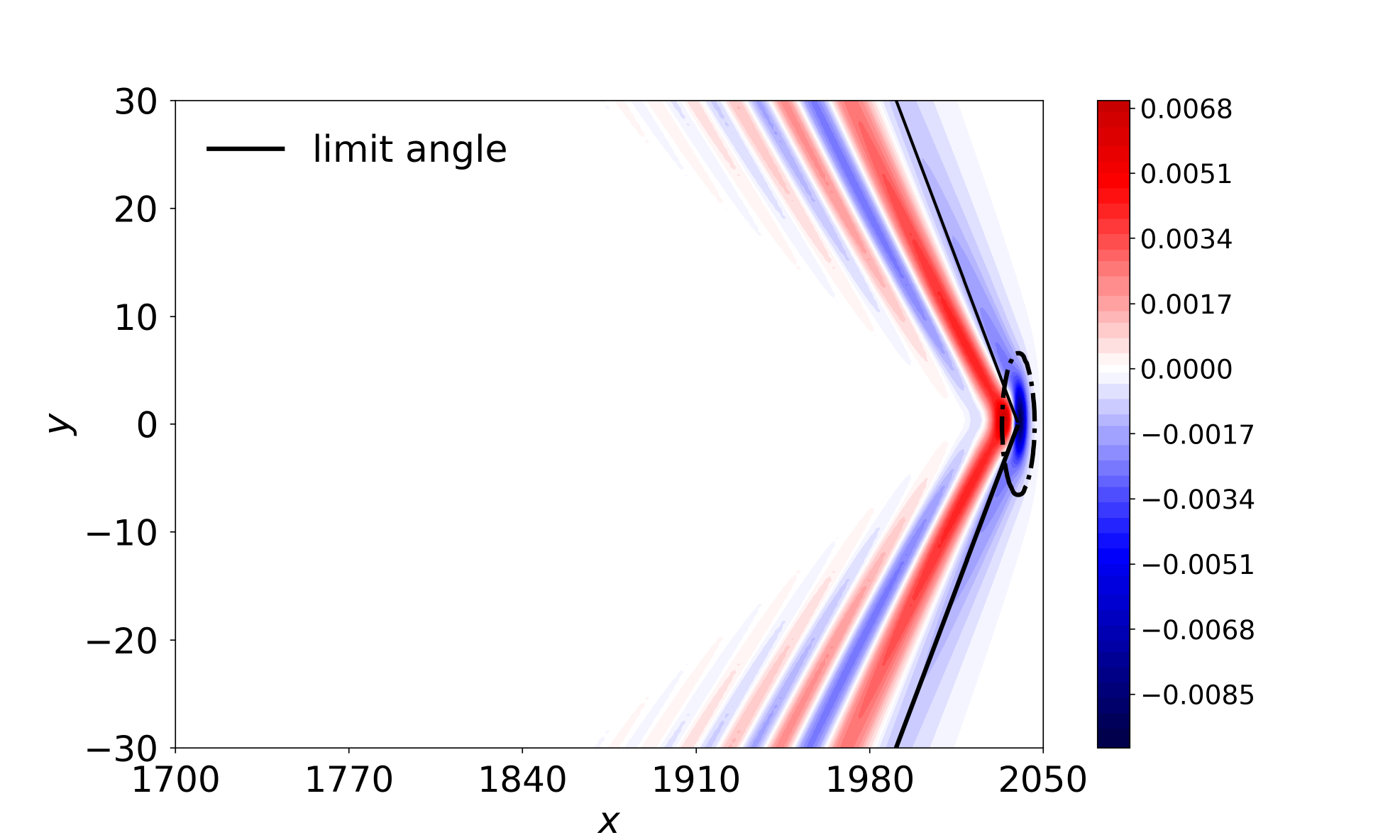}
    \end{minipage}
   \caption{Example~\ref{ex:2Dwaves}. Interface deformation  $\eta_{\varepsilon}^N$ at the final time  $t^N = 2500~\si{s}$ in the supercritical regime: the two-dimensional interface is shown on the left, while a zoomed view around the boat is displayed on the right, where the black line indicates the limiting angle $\varphi^{\star}$ defined in equation~\eqref{varphi_star}.}
    \label{fig:cas_super_crit_2dmu0est0}
\end{figure}
\FloatBarrier

\vspace{0.5em}
A similar simulation is carried out in the supercritical regime, as illustrated in Figure~\ref{fig:cas_super_crit_2dmu0est0}, which shows the interface deformation at the final time.
The two-dimensional view (left) highlights a wake persisting behind the ship.

In contrast with the subcritical case, this wake extends over a long distance downstream.
The zoomed-in view (right) shows that the amplitude of the deformation is weaker than in the subcritical regime. As described in section~\ref{subsec:2D_theory}, there are only divergent waves and no transverse waves. Moreover, the wake is contained within a cone centered on the ship that form an angle $\varphi^{\star} \approx 0.547~\si{rd}$ (in radian, obtained with~\eqref{varphi_star}) with the $\bm{e_x}-$axis. This cone is delimited by black straight lines in Figure~\ref{fig:cas_super_crit_2dmu0est0}.

\section{Conclusion and Perspectives}
In conclusion, this paper analyzes and models the phenomenon of dead water in a bi-layer environment, highlighting the formation of internal waves caused by the passage of a ship. One- and two-dimensional simulations, based on the spectral method and the exponential integrator, were conducted to evaluate the accuracy and stability of the obtained solutions. The parameter $\varepsilon$ plays a key role in controlling numerical oscillations, with its adjustment in the subcritical and supercritical regimes leading to a more stable and regular surface. One-dimensional simulations have shown that the phenomenon is more pronounced at low velocity, leading to an over-elevation under the ship's hull and the propagation of waves at the front and back of the ship. The existence of a critical navigation speed has also been demonstrated, influencing wave dynamics at the interface. In two dimensions, we deduce that the structure of the wake depends on the nature of the regime. In the subcritical regime, the waves have a marked amplitude but remain confined to a short distance. In contrast, in the supercritical regime, they have a lower amplitude but persist over a long distance downstream. Finally, the results obtained are consistent with theoretical predictions, confirming the accuracy of the methods used. 

This study provides several perspectives, for instance a more advanced analysis of the parameters influencing the oscillations, such as a traction force, to better understand their impact on internal waves. In addition, optimizing the boat's speed could improve its performance and limit the effect of dead water.

\appendix
\section{Computation of an optimal regularization parameter}
\label{app:exponentialintegrator}
In this paper, an exponential integrator~\eqref{eq:ERK1_WW} is used to solve~\eqref{eq:WaterWaves_viscous}, offering good accuracy with the advantage of allowing larger time steps than explicit methods. The parameter $\varepsilon$ corresponds to an artificial damping that ensures existence of a solution as seen in Theorem~\ref{thm:wellposed}.  It dissipates parasitic modes coming from Fourier truncation. In this section, we detail our strategy to compute $\varepsilon$ as small as possible so as to minimize parasitic oscillations without deteriorating the physical phenomenons.

Lets define $\Moperator(\varepsilon) = \begin{vmatrix} \eta_{\varepsilon}(x_{\max}) - \eta_{\varepsilon}(x_{\min}) \end{vmatrix}$ with $x_{\max} = \arg\max \eta_{\kfourier,\varepsilon}(x, t^N)$ and $x_{\min} = \arg\min \eta_{\kfourier,\varepsilon}(x, t^N)$ are the positions of the first detected maximum and minimum oscillations in the interval in front of the boat and at final time $t^N$. The parameter $\varepsilon$ is optimized to reduce the oscillation's measurement $\Moperator(\varepsilon)$. We define a tolerance $\delta>0$ and we look for a parameter $\varepsilon^*$ such that: 
\begin{equation*}
    \varepsilon^* = \min \left\lbrace \varepsilon  \geq 0 \mid \Moperator(\varepsilon) < \delta \right\rbrace.
\end{equation*}
In practice $\varepsilon^*$ is computed by induction on $n \in \mathbb{N}$. Starting with $\varepsilon_n=\varepsilon_0$, we have two possibilities:
\begin{itemize}
    \item if $M(\varepsilon_n) < \delta$, the optimal value is reached and we define $\varepsilon^* = \varepsilon_n$,
    \item else, compute $\varepsilon_{n+1} = \gamma \cdot \varepsilon_{n}$ and repeat the process with the new iterate. Here, $\gamma$ represents a factor allowing for progressive adjustments, typically  set to $\gamma = 1.1$.
\end{itemize}
This process continues until the convergence condition is satisfied, indicating that the oscillations have been sufficiently reduced. Given a final simulation time $t^N > 0$ and physical parameters, the global algorithm to compute $\varepsilon$ is outlined in Algorithm~\ref{alg:varepsilon}.

\begin{algorithm}[ht]
\caption{Computation of the damping parameter $\varepsilon$.}
\label{alg:varepsilon}
\begin{algorithmic}[1]
    \REQUIRE $\varepsilon_0$, $\delta = 10^{-7}$, $\gamma=1.1$.
    \FOR{$n = 0, \ldots,$ until convergence}
        \STATE Compute $\mudiff_{\kfourier,\varepsilon_n}(t^N)$ using~\eqref{eq:ERK1_WW};
        \STATE Compute the interface at final time $t^N$ using~\eqref{eq:eta_varphi} and the inverse Fourier transform    $$\eta_{\varepsilon_n}^N =  \mathcal{F}^{-1} \left( \dfrac{\mudiff_{\kfourier,\varepsilon_n}+\overline{\mudiff_{-\kfourier,\varepsilon_n}}}{2} \right);$$
        \STATE Compute $\Moperator(\varepsilon_n) = \begin{vmatrix} \eta_{\varepsilon_n}^N(x_{\text{max}}) - \eta_{\varepsilon_n}^N(x_{\text{min}}) \end{vmatrix}$;
        \IF{$\Moperator(\varepsilon_n) < \delta$}
            \STATE $\varepsilon^* = \varepsilon_n$;
            \STATE \textbf{break}
        \ELSE
            \STATE $\varepsilon_{n+1} = \gamma \cdot \varepsilon_n$.
        \ENDIF
    \ENDFOR
\end{algorithmic}
\end{algorithm}

\section*{Acknowledgements}
The authors are thankful to Julien Dambrine for helpful discussions. 

This work pertains (namely is not funded but enters in the corresponding scientific perimeter) to the French government program ``Investissements d'Avenir'' (LABEX INTERACTIFS, reference ANR-11-LABX-0017-01 and EUR INTREE, reference ANR-18-EURE-001). 

\bibliographystyle{abbrv}
\bibliography{References.bib}

\end{document}